\documentclass[preprint,12pt]{elsarticle}

\usepackage{amssymb}
 \usepackage{graphicx}
\usepackage{epsfig}

\usepackage{amsthm}

\newcommand{\sysn}{\left\{\begin{array}{rcl}}
\newcommand{\sysk}{\end{array}\right.}

\newcommand{\ingrw}[2]{\includegraphics[width=#1mm]{#2}}

\newtheorem{theorem}{Theorem}[section]

\theoremstyle{example}

\newtheorem{proposition}[theorem]{Proposition}
\theoremstyle{definition}
\newtheorem{definition}[theorem]{Definition}

\newtheorem{corollary}[theorem]{Corollary}

\journal{Topology and its Applications}

\begin{document}

\begin{frontmatter}



\title{Classification of selectors for sequences of dense sets of $C_p(X)$ \tnoteref{label1}}


\author{Alexander V. Osipov}

\ead{OAB@list.ru}


\address{Krasovskii Institute of Mathematics and Mechanics, Ural Federal
 University,

 Ural State University of Economics, Yekaterinburg, Russia}

\begin{abstract} For a Tychonoff space $X$, we denote by $C_p(X)$
the space of all real-valued continuous functions on $X$ with the
topology of pointwise convergence. In this paper we investigate
different selectors for sequences of dense sets of $C_p(X)$. We
give the characteristics of selection principles
$S_{1}(\mathcal{P},\mathcal{Q})$,
$S_{fin}(\mathcal{P},\mathcal{Q})$ and
$U_{fin}(\mathcal{P},\mathcal{Q})$ for $\mathcal{P},\mathcal{Q}\in
\{\mathcal{D}$, $\mathcal{S}$, $\mathcal{A} \}$, where

$\bullet$ $\mathcal{D}$ --- the family of a  dense subsets  of
$C_p(X)$;

$\bullet$ $\mathcal{S}$ --- the family of a sequentially dense
subsets  of $C_p(X)$;

$\bullet$ $\mathcal{A}$ --- the family of a  $1$-dense subsets of
$C_p(X)$, through the selection principles of a space $X$.


\end{abstract}


\end{frontmatter}



\section{Introduction}

Throughout this paper, all spaces are assumed to be Tychonoff. The
set of positive integers is denoted by $\omega$. Let $\mathbb{R}$
be the real line, we put $\mathbb{I}=[0,1]\subset \mathbb{R}$, and
$\mathbb{Q}$ be the rational numbers. For a space $X$, we denote
by $C_p(X)$ the space of all real-valued continuous functions on
$X$ with the topology of pointwise convergence. The symbol
$\bf{0}$ stands for the constant function to $0$.

Basic open sets of $C_p(X)$ are of the form

$[x_1,...,x_k, U_1,...,U_k]=\{f\in C(X): f(x_i)\in U_i$,
$i=1,...,k\}$, where each $x_i\in X$ and each $U_i$ is a non-empty
open subset of $\mathbb{R}$. Sometimes we will write the basic
neighborhood of the point $f$ as $<f,A,\epsilon>$ where
$<f,A,\epsilon>:=\{g\in C(X): |f(x)-g(x)|<\epsilon$ $\forall x\in
A\}$, $A$ is a finite subset of $X$ and $\epsilon>0$.

In this paper, by cover we mean a nontrivial one, that is,
$\mathcal{U}$ is a cover of $X$ if $X=\cup \mathcal{U}$ and
$X\notin \mathcal{U}$.

 An open cover $\mathcal{U}$ of a space $X$ is:

 $\bullet$ an {\it $\omega$-cover} if $X$ does not belong to
 $\mathcal{U}$ and every finite subset of $X$ is contained in a
 member of $\mathcal{U}$.

$\bullet$ a {\it $\gamma$-cover} if it is infinite and each $x\in
X$ belongs to all but finitely many elements of $\mathcal{U}$.

For a topological space $X$ we denote:

$\bullet$ $\mathcal{O}$ --- the family of open covers of $X$;


$\bullet$ $\Gamma$ --- the family of open $\gamma$-covers of $X$;



$\bullet$ $\Gamma_{cl}$ --- the family of clopen $\gamma$-covers
of $X$;

$\bullet$ $\Omega$ --- the family of open $\omega$-covers of $X$;



$\bullet$ $\mathcal{D}$ --- the family of a  dense subsets of $X$;

$\bullet$ $\mathcal{S}$ --- the family of a sequentially dense
subsets of $X$.

\bigskip

Many topological properties are defined or characterized in terms
 of the following classical selection principles.
 Let $\mathcal{A}$ and $\mathcal{B}$ be sets consisting of
families of subsets of an infinite set $X$. Then:

$S_{1}(\mathcal{A},\mathcal{B})$ is the selection hypothesis: for
each sequence $\{A_{n}: n\in \omega\}$ of elements of
$\mathcal{A}$ there is a sequence $\{b_{n}\}_{n\in \omega}$ such
that for each $n$, $b_{n}\in A_{n}$, and $\{b_{n}: n\in\omega \}$
is an element of $\mathcal{B}$.

$S_{fin}(\mathcal{A},\mathcal{B})$ is the selection hypothesis:
for each sequence $\{A_{n}: n\in \omega\}$ of elements of
$\mathcal{A}$ there is a sequence $\{B_{n}\}_{n\in \omega}$ of
finite sets such that for each $n$, $B_{n}\subseteq A_{n}$, and
$\bigcup_{n\in\omega}B_{n}\in\mathcal{B}$.

$U_{fin}(\mathcal{A},\mathcal{B})$ is the selection hypothesis:
whenever $\mathcal{U}_1$, $\mathcal{U}_2, ... \in \mathcal{A}$ and
none contains a finite subcover, there are finite sets
$\mathcal{F}_n\subseteq \mathcal{U}_n$, $n\in \omega$, such that
$\{\bigcup \mathcal{F}_n : n\in \omega\}\in \mathcal{B}$.

\medskip
The following prototype of many classical properties is called
"$\mathcal{A}$ choose $\mathcal{B}$" in \cite{tss}.

${\mathcal{A}\choose\mathcal{B}}$ : For each $\mathcal{U}\in
\mathcal{A}$ there exists $\mathcal{V}\subseteq \mathcal{U}$ such
that $\mathcal{V}\in \mathcal{B}$.

Then $S_{fin}(\mathcal{A},\mathcal{B})$ implies
${\mathcal{A}\choose\mathcal{B}}$.

Many equivalence hold among these properties, and the surviving
ones appear in the following Diagram (where an arrow denote
implication), to which no arrow can be added except perhaps from
$U_{fin}(\Gamma, \Gamma)$ or $U_{fin}(\Gamma, \Omega)$ to
$S_{fin}(\Gamma, \Omega)$ \cite{jmss}.

\bigskip

\begin{center}
\ingrw{90}{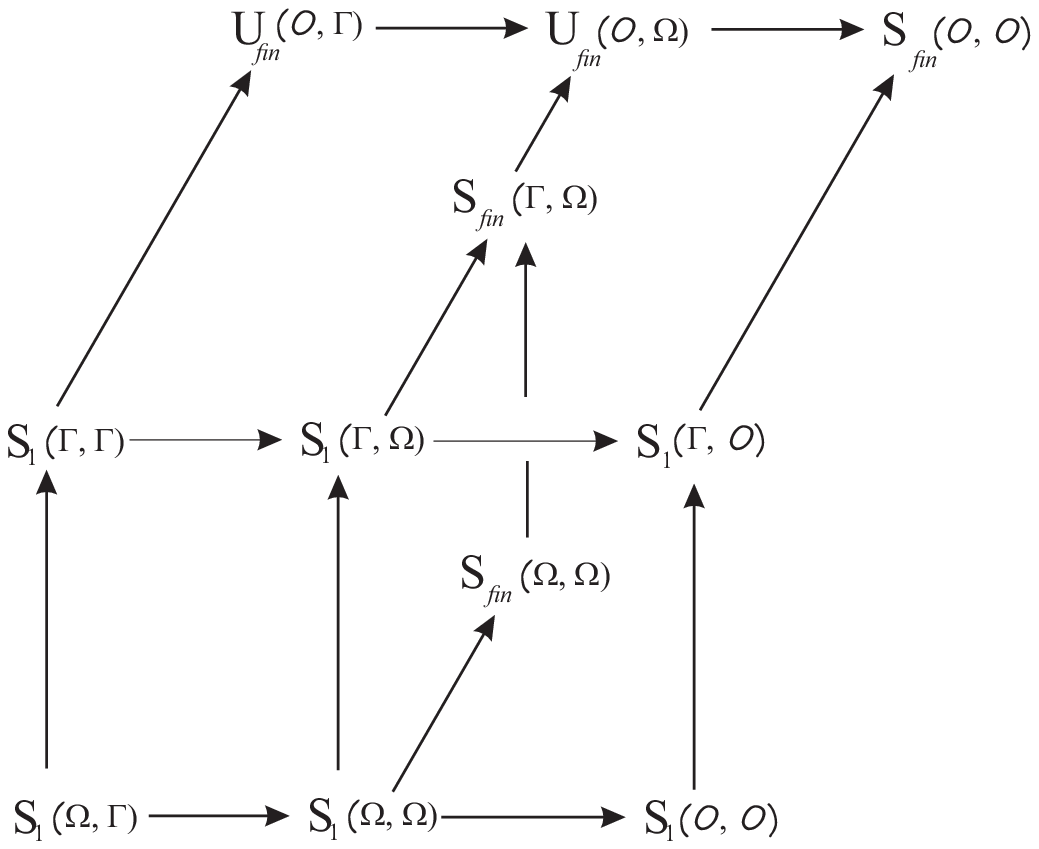}

\medskip

Fig.~1. The Scheepers Diagram.

\end{center}
\bigskip

The papers \cite{jmss,ko,sch3,sch1,bts} have initiated the
simultaneous
 consideration of these properties in the case where $\mathcal{A}$ and
 $\mathcal{B}$ are important families of open covers of a
 topological space $X$.

\section{Main definitions and notation}

Let $X$ be a topological space, and $x\in X$. A subset $A$ of $X$
{\it converges} to $x$, $x=\lim A$, if $A$ is infinite, $x\notin
A$, and for each neighborhood $U$ of $x$, $A\setminus U$ is
finite. Consider the following collection:

$\bullet$ $\Omega_x=\{A\subseteq X : x\in \overline{A}\setminus
A\}$;

$\bullet$ $\Gamma_x=\{A\subseteq X : x=\lim A\}$.

\bigskip

We write $\Pi (\mathcal{A}_x, \mathcal{B}_x)$ without specifying
$x$, we mean $(\forall x) \Pi (\mathcal{A}_x, \mathcal{B}_x)$.

$\bullet$ A space $X$ has {\it countable fan tightness}
(Arhangel'skii's countable fan tightness), if $X$ $\models$
$S_{fin}(\Omega_x,\Omega_x)$ \cite{arh}.

\medskip

$\bullet$ A space $X$ has {\it countable strong fan tightness}
(Sakai's countable strong fan tightness), if $X$ $\models$
$S_{1}(\Omega_x,\Omega_x)$ \cite{sak}.

\medskip

$\bullet$ A space $X$ has {\it countable selectively sequentially
fan tightness} (Arhangel'skii's property $\alpha_4$), if $X$
$\models$ $S_{fin}(\Gamma_x,\Gamma_x)$ \cite{arh0}.

$\bullet$  A space $X$ has {\it countable strong selectively
sequentially fan tightness} (Arhangel'skii's property $\alpha_2$),
if $X$ $\models$ $S_{1}(\Gamma_x,\Gamma_x)$ \cite{arh0}.

\medskip

$\bullet$ A space $X$ has {\it strictly Fr$\acute{e}$chet-Urysohn
at $x$,} if $X$ $\models$ $S_{1}(\Omega_x,\Gamma_x)$ \cite{sash}.

\medskip

$\bullet$ A space $X$ has {\it almost strictly
Fr$\acute{e}$chet-Urysohn at $x$},  if $X$ $\models$
$S_{fin}(\Omega_x,\Gamma_x)$.

\medskip
$\bullet$ A space $X$ has {\it the weak sequence selection
property}, if $X$ $\models$ $S_{1}(\Gamma_x,\Omega_x)$
\cite{sch4}.

\medskip

$\bullet$  A space $X$ has {\it the sequence selection property},
if $X$ $\models$ $S_{fin}(\Gamma_x,\Omega_x)$.

\medskip

\medskip
The following implications hold

\medskip
\begin{center}
$S_1(\Gamma_x,\Gamma_x) \Rightarrow S_{fin}(\Gamma_x,\Gamma_x)
\Rightarrow S_1(\Gamma_x,\Omega_x) \Rightarrow
S_{fin}(\Gamma_x,\Omega_x)$ \\  $\Uparrow$ \, \,\, \, \, \, \,
\,\, \, \, \, \, \, $ \Uparrow $ \,\, \, \, \,\, \, \, \, \, \,
$\Uparrow $ \,\, \, \, \,\,\,\, \, \, \, \,\, \, \, $\Uparrow$ \\
$S_1(\Omega_x,\Gamma_x) \Rightarrow S_{fin}(\Omega_x,\Gamma_x)
\Rightarrow S_1(\Omega_x,\Omega_x) \Rightarrow
S_{fin}(\Omega_x,\Omega_x)$

\end{center}

\medskip

We write $\Pi (\mathcal{A}, \mathcal{B}_x)$ without specifying
$x$, we mean $(\forall x) \Pi (\mathcal{A}, \mathcal{B}_x)$.
\medskip

\medskip

$\bullet$  A space $X$ has {\it countable fan tightness with
respect to dense subspaces,} if $X$ $\models$
$S_{fin}(\mathcal{D},\Omega_x)$ (\cite{bbm1}).

$\bullet$ A space $X$ has {\it countable strong fan tightness with
respect to dense subspaces,} if $X$ $\models$
$S_{1}(\mathcal{D},\Omega_x)$ (\cite{bbm1}).

$\bullet$ A space $X$ has {\it almost strictly
Fr$\acute{e}$chet-Urysohn at $x$ with respect to dense subspaces},
if $X$ $\models$ $S_{fin}(\mathcal{D},\Gamma_x)$.

$\bullet$ A space $X$ has {\it strictly Fr$\acute{e}$chet-Urysohn
at $x$ with respect to dense subspaces}, if $X$ $\models$
$S_{1}(\mathcal{D},\Gamma_x)$.

$\bullet$ A space $X$ has {\it countable selectively sequentially
fan tightness with respect to dense subspaces}, if $X$ $\models$
$S_{fin}(\mathcal{S},\Gamma_x)$.

$\bullet$ A space $X$ has {\it countable strong selectively
sequentially fan tightness with respect to dense subspaces}, if
$X$ $\models$ $S_{1}(\mathcal{S},\Gamma_x)$.

$\bullet$ A space $X$ has {\it the sequence selection property
with respect to dense subspaces}, if $X$ $\models$
$S_{fin}(\mathcal{S},\Omega_x)$.

$\bullet$ A space $X$ has {\it the weak sequence selection
property with respect to dense subspaces}, if $X$ $\models$
$S_{1}(\mathcal{S},\Omega_x)$.


\medskip

\medskip
The following implications hold

\medskip
\begin{center}
$S_1(\mathcal{S},\Gamma_x) \Rightarrow
S_{fin}(\mathcal{S},\Gamma_x) \Rightarrow
S_1(\mathcal{S},\Omega_x) \Rightarrow
S_{fin}(\mathcal{S},\Omega_x)$ \\  \, \, \, $\Uparrow$ \, \, \, \,
\,\,  \, \, \, \, \, $ \Uparrow $ \,\, \, \, \,\, \, \, \, \, \,
$\Uparrow $ \,\, \, \, \,\, \, \,\, \, \, $\Uparrow$
\\ $S_1(\mathcal{D},\Gamma_x) \Rightarrow
S_{fin}(\mathcal{D},\Gamma_x) \Rightarrow
S_1(\mathcal{D},\Omega_x) \Rightarrow
S_{fin}(\mathcal{D},\Omega_x)$

\end{center}

\medskip
\medskip

$\bullet$ A space $X$ is $R$-separable, if $X$ $\models$
$S_1(\mathcal{D}, \mathcal{D})$ (Def. 47, \cite{bbm1}).

$\bullet$ A space $X$ is $M$-separable (selective separability),
if $X$ $\models$ $S_{fin}(\mathcal{D}, \mathcal{D})$.

$\bullet$ A space $X$ is selectively sequentially separable, if
$X$ $\models$ $S_{fin}(\mathcal{S}, \mathcal{S})$ (Def. 1.2,
\cite{bc}).

\medskip
\medskip
The following implications hold

\medskip
\begin{center}
$S_1(\mathcal{S},\mathcal{S}) \Rightarrow
S_{fin}(\mathcal{S},\mathcal{S}) \Rightarrow
S_1(\mathcal{S},\mathcal{D}) \Rightarrow
S_{fin}(\mathcal{S},\mathcal{D})$ \\  \, \, $\Uparrow$ \, \, \, \,
\,\, \,  \, \, \, $ \Uparrow $ \,\, \, \, \,\, \, \,
\, \, \, $\Uparrow $ \,\, \, \, \,\, \, \,\, \, \, $\Uparrow$ \\
$S_1(\mathcal{D},\mathcal{S}) \Rightarrow
S_{fin}(\mathcal{D},\mathcal{S}) \Rightarrow
S_1(\mathcal{D},\mathcal{D}) \Rightarrow
S_{fin}(\mathcal{D},\mathcal{D})$

\end{center}

\medskip

  If $X$ is a  space and $A\subseteq X$, then the sequential closure of $A$,
 denoted by $[A]_{seq}$, is the set of all limits of sequences
 from $A$. A set $D\subseteq X$ is said to be sequentially dense
 if $X=[D]_{seq}$. If $D$ is a countable sequentially dense subset
 of $X$ then $X$ call sequentially separable space.

 Call $X$ strongly sequentially dense in itself, if every dense subset of $X$ is sequentially dense, and, $X$ strongly sequentially separable, if $X$ is separable and
 every countable dense subset of $X$ is sequentially dense.
 Clearly, every strongly sequentially separable space is
 sequentially separable, and every sequentially separable space is
 separable.

 We recall that a subset of $X$ that is the
 complete preimage of zero for a certain function from~$C(X)$ is called a zero-set.
A subset $O\subseteq X$  is called  a cozero-set (or functionally
open) of $X$ if $X\setminus O$ is a zero-set.

\medskip
Recall that the $i$-weight $iw(X)$ of a space $X$ is the smallest
infinite cardinal number $\tau$ such that $X$ can be mapped by a
one-to-one continuous mapping onto a Tychonoff space of the weight
not greater than $\tau$.

\medskip

\medskip

\begin{theorem} (Noble \cite{nob}) \label{th31} Let $X$ be a   space. A space $C_{p}(X)$ is separable if and only if
$iw(X)=\aleph_0$.
\end{theorem}

\begin{definition} A space $X$ has {\bf $V$-property} ($X$ $\models$ $V$), if there
 exist  a condensation (one-to-one continuous mapping) $f: X \mapsto Y$ from the space $X$ on a
 separable metric space $Y$, such that $f(U)$ --- $F_{\sigma}$-set
 of $Y$ for any cozero-set $U$ of $X$.
\end{definition}

\begin{theorem} \label{th38} (Velichko \cite{vel}). Let $X$ be a Tychonoff space. A space $C_p(X)$ is
sequentially separable if and only if  $X$ $\models$ $V$.
\end{theorem}
\medskip

Recall that the cardinal $\mathfrak{p}$ is the smallest cardinal
so that there is a collection of $\mathfrak{p}$ many subsets of
the natural numbers with the strong finite intersection property
but no infinite pseudo-intersection. Note that $\omega_1 \leq
\mathfrak{p} \leq \mathfrak{c}$.

For $f,g\in \mathbb{N}^{\mathbb{N}}$, let $f\leq^{*} g$ if
$f(n)\leq g(n)$ for all but finitely many $n$. $\mathfrak{b}$ is
the minimal cardinality of a $\leq^{*}$-unbounded subset of
$\mathbb{N}^{\mathbb{N}}$. A set $B\subset [\mathbb{N}]^{\infty}$
is unbounded if the set of all increasing enumerations of elements
of $B$ is unbounded in $\mathbb{N}^{\mathbb{N}}$, with respect to
$\leq^{*}$. It follows that $|B|\geq \mathfrak{b}$. A subset $S$
of the real line is called a $Q$-set if each one of its subsets is
a $G_{\delta}$. The cardinal $\mathfrak{q}$ is the smallest
cardinal so that for any $\kappa< \mathfrak{q}$ there is a $Q$-set
of size $\kappa$. (See \cite{do} for more on small cardinals
including $\mathfrak{p}$).



\medskip

\section{$S_{1}(\mathcal{D},\mathcal{D})$ --- $R$-separable}

\medskip

In \cite{sak} (Lemma, Theorem 1), M. Sakai proved:

\begin{theorem} $(Sakai)$ For each   space $X$ the
following are equivalent.
\begin{enumerate}

\item $C_p(X)$ $\models$ $S_{1}(\Omega_x,\Omega_x)$.

\item $X^n$ $\models$ $S_{1}(\mathcal{O},\mathcal{O})$ $(X^n$ has
Rothberger's property $C^{''})$ for each $n\in \omega$.

\item $X$ $\models$ $S_1(\Omega, \Omega)$.

\end{enumerate}

\end{theorem}

 In (\cite{sch}, Theorem 13) M. Scheeper was proved the following result

\begin{theorem}$(Scheeper)$ \label{th21} For each separable metric space $X$, the
following are equivalent:

\begin{enumerate}

\item $C_p(X)$ $\models$  $S_{1}(\mathcal{D},\mathcal{D})$;

\item $X$ $\models$ $S_{1}(\Omega, \Omega)$.

\end{enumerate}

\end{theorem}

\medskip

By Theorem 57 in \cite{bbm1}, \cite{sak} and Theorem \ref{th31},
we have

\begin{theorem}\label{th11} For a   space $X$, the
following are equivalent:

\begin{enumerate}

\item $C_p(X)$ $\models$ $S_{1}(\mathcal{D},\mathcal{D})$;

\item $C_p(X)$ $\models$ $S_{1}(\Omega_x,\Omega_x)$, and is
separable;

\item $C_p(X)$ $\models$ $S_{1}(\mathcal{D},\Omega_x)$, and is
separable;

\item $X$ $\models$ $S_{1}(\Omega, \Omega)$, and $iw(X)=\aleph_0$;

\item $X^n$ $\models$ $S_{1}(\mathcal{O},\mathcal{O})$ for each
$n\in \omega$, and $iw(X)=\aleph_0$.

\end{enumerate}

\end{theorem}

\begin{corollary} For a separable metrizable space $X$, the
following are equivalent:

\begin{enumerate}

\item $C_p(X)$ $\models$ $S_{1}(\mathcal{D},\mathcal{D})$
[$R$-separable];

\item $C_p(X)$ $\models$ $S_{1}(\Omega_x,\Omega_x)$ [countable
strong fan tightness];

\item $C_p(X)$ $\models$ $S_{1}(\mathcal{D},\Omega_x)$ [countable
strong fan tightness with respect to dense subspaces];

\item $X$ $\models$ $S_{1}(\Omega, \Omega)$;

\item $X^n$ $\models$ $S_{1}(\mathcal{O},\mathcal{O})$ for each
$n\in \omega$ [$X^n$ is Rothberger].

\end{enumerate}

\end{corollary}

\section{$S_{fin}(\mathcal{D},\mathcal{D})$ --- $M$-separable}

 In (\cite{arh}, Theorem 2.2.2 in \cite{arh2}) A.V. Arhangel'skii was proved the following result

\begin{theorem} $(Arhangel'skii)$ \label{th46} For a   space $X$, the
following are equivalent:

\begin{enumerate}

\item $C_p(X)$ $\models$ $S_{fin}(\Omega_x, \Omega_x)$;

 \item $(\forall n\in \omega)$ $X^{n}$ $\models$ $S_{fin}(\mathcal{O},
\mathcal{O})$.

\end{enumerate}

\end{theorem}

It is known (see \cite{jmss}) that $X$ $\models$ $S_{fin}(\Omega,
\Omega)$ iff $(\forall n\in \omega)$ $X^{n}$ $\models$
$S_{fin}(\mathcal{O}, \mathcal{O})$.

\medskip

By Theorem 21 in \cite{bbm1} and Theorem 3.9 in \cite{jmss}, we
have a next result.

\begin{theorem} \label{th44} For a   space $X$, the
following are equivalent:

\begin{enumerate}

\item $C_p(X)$ $\models$ $S_{fin}(\mathcal{D},\mathcal{D})$;

\item $X$ $\models$ $S_{fin}(\Omega, \Omega)$ and
$iw(X)=\aleph_0$;

\item $(\forall n\in \omega)$ $X^{n}\in S_{fin}(\mathcal{O},
\mathcal{O})$ and $iw(X)=\aleph_0$;

 \item $C_p(X)$ $\models$ $S_{fin}(\Omega_x, \Omega_x)$ and is  separable;

\item $C_p(X)$ $\models$ $S_{fin}(\mathcal{D}, \Omega_x)$ and is
separable.

\end{enumerate}

\end{theorem}

\begin{corollary} For a separable metrizable space $X$, the
following are equivalent:

\begin{enumerate}

\item $C_p(X)$ $\models$ $S_{fin}(\mathcal{D},\mathcal{D})$
[$M$-separable];

\item $C_p(X)$ $\models$ $S_{fin}(\Omega_x,\Omega_x)$ [countable
fan tightness];

\item $C_p(X)$ $\models$ $S_{fin}(\mathcal{D},\Omega_x)$
[countable strong fan tightness with respect to dense subspaces];

\item $X$ $\models$ $S_{fin}(\Omega, \Omega)$;

\item $X^n$ $\models$ $S_{fin}(\mathcal{O},\mathcal{O})$ for each
$n\in \omega$ [$X^n$ is Menger].

\end{enumerate}

\end{corollary}

\section{$S_{1}(\mathcal{D},\mathcal{S})$}

 Pytkeev \cite{py} and independently Gerlits \cite{ger1}, see also
\cite{arh2} and \cite{mcnt}, proved

\begin{theorem} \label{th80} For a  space $X$, the following statements are
equivalent:

\begin{enumerate}

\item $C_p(X)$ is Fr$\acute{e}$chet-Urysohn;

\item $C_p(X)$ is sequential;

\item $C_p(X)$ is a $k$-space.

\end{enumerate}

\end{theorem}

 Gerlits and Nagy \cite{gn} proved

\begin{theorem}(Gerlits, Nagy) \label{th7} For a space $X$, the following statements are
equivalent:

\begin{enumerate}

\item $C_p(X)$ $\models$ $S_{1}(\Omega_x, \Gamma_x)$;

\item $C_p(X)$ is Fr$\acute{e}$chet-Urysohn;

\item $X$ $\models$ $S_{1}(\Omega, \Gamma)$;

\item $X$ $\models$ ${\Omega\choose\Gamma}$.

\end{enumerate}

\end{theorem}

\begin{theorem} \label{th71} For a space $X$, the following statements are
equivalent:

\begin{enumerate}

\item $C_p(X)$ is strongly sequentially dense in itself;

\item $X$ $\models$ $S_{1}(\Omega, \Gamma)$.

\end{enumerate}

\end{theorem}

\begin{proof} $(1)\Rightarrow(2)$. By Theorem \ref{th7}, $S_{1}(\Omega,
\Gamma)={\Omega\choose\Gamma}$.
 Let $\mathcal{U}\in
\Omega$ and $P$ be a dense subset of $C_p(X)$. A set $\mathcal{D}
:=\{ f\in C(X) : f\upharpoonright K=h$ for $h\in P$, and
$f\upharpoonright (X\setminus U)=1$ for a finite subset $K\subset
U$ where $U\in \mathcal{U} \}$.

Since $\mathcal{U}$ is a $\omega$-cover of $X$ and $P$ is a dense
subset of $C_p(X)$, we claim that $\mathcal{D}$ is a dense subset
of $C_p(X)$.

Fix $g\in C(X)$. Let $K$ be a finite subset of $X$, $\epsilon>0$
and $W=<g, K,\epsilon>$ be a base neighborhood of $g$, then there
is $U\in \mathcal{U}$ such that $K\subset U$ and $h\in W$ for some
$h\in P$. Since $f\upharpoonright K= h\upharpoonright K$ for some
$f\in \mathcal{D}$, then $f\in W$. By (1), $\mathcal{D}$ is a
sequentially dense subset of $C_p(X)$. Then there exists a
sequence $\{f_{i}\}_{i\in \omega}$ such that for each $i$,
$f_{i}\in \mathcal{D}$, and $\{f_{i}\}_{i\in\omega}$ converge to
$\bf{0}$. By definition of $f_i$, $f_i\upharpoonright K_i=h_{i}$
for some finite set $K_i$ and $h_{i}\in P$, and
$f_i\upharpoonright (X\setminus U_i)=1$ for some $U_i\in
\mathcal{U}$.

Consider a sequence $\{U_{i}\}_{i\in \omega}$.

(a). $U_{i}\in \mathcal{U}$.

(b). $\{U_{i}: i\in \omega\}$ is a $\gamma$-cover of $X$.

Let $K$ be a finite subset of $X$ and $W=[K,(-1,1)]$ be a base
neighborhood of $\bf{0}$, then there is $i'\in \omega$ such that
$f_{i}\in W$ for each $i>i'$. It follows that $K\subset U_i$ for
each $i>i'$. We thus get that $X$ $\models$
${\Omega\choose\Gamma}$, and, hence, $X$ $\models$ $S_{1}(\Omega,
\Gamma)$.

$(2)\Rightarrow(1)$. By Theorem \ref{th7}, $C_p(X)$ is
Fr$\acute{e}$chet-Urysohn, and, hence, $C_p(X)$ is strongly
sequentially dense in itself.
\end{proof}

\begin{corollary} For a   space $X$, the following statements are
equivalent:

\begin{enumerate}

\item $C_p(X)$ $\models$ $S_{1}(\Omega_x, \Gamma_x)$;

\item $C_p(X)$ is Fr$\acute{e}$chet-Urysohn;

\item $X$ $\models$ $S_{1}(\Omega, \Gamma)$;

\item $X$ $\models$ ${\Omega\choose\Gamma}$;

\item $C_p(X)$ is sequential;

\item $C_p(X)$ is a $k$-space;

\item $C_p(X)$ is strongly sequentially dense in itself.

\end{enumerate}

\end{corollary}

Note that $S_{1}(\Omega, \Gamma)=S_{fin}(\Omega, \Gamma)$ (see
\cite{jmss}).

\begin{theorem}\label{th44} For a   space $X$, the following statements are
equivalent:

\begin{enumerate}

\item $X$ $\models$ $S_{fin}(\Omega, \Gamma)$;

\item $C_p(X)$ $\models$ $S_{fin}(\Omega_x, \Gamma_x)$.

\end{enumerate}

\end{theorem}

\begin{proof} By Theorem \ref{th7}, it suffices to prove $(2)\Rightarrow(1)$.

$(2)\Rightarrow(1)$. Let $\{\mathcal{U}_n\}_{n\in \omega}$ be a
sequence of open $\omega$-covers of $X$. We set $A_n=\{f\in C(X):
f\upharpoonright (X\setminus U)=0$ for $U\in \mathcal{U}_n \}$. It
is not difficult to see that each $A_n$ is dense in $C(X)$ since
each $\mathcal{U}_n$ is an $\omega$-cover of $X$ and $X$ is
Tychonoff. Let $f$ be the constant function to 1. By the
assumption there exist $\{f^i_n : i=1,...,k(n)\}\subset A_n$ such
that $\bigcup_{n\in \omega} \{f^i_n\}_{i=1}^{k(n)}$ converge to
$f$. Consider subsequence $\{f^1_n\}_{n\in \omega}\subset
\bigcup_{n\in \omega} \{f^i_n\}_{i=1}^{k(n)}$. Note that
$\{f^1_n\}_{n\in \omega}$ also converge to $f$.

 For each $f^1_n$ we
take $U_n\in \mathcal{U}_n$ such that
$f^1_n\upharpoonright(X\setminus U_n)=0$.

 Set $\mathcal{U}=\{ U_n : n\in \omega\}$. For each finite subset
$\{x_1,...,x_k\}$ of $X$ we consider the basic open neighborhood
of $f$ $[x_1,...,x_k; W,..., W]$, where $W=(0,2)$.

 Note that there is $n'\in \omega$ such that
$[x_1,...,x_k; W,..., W]$ contains $f^1_n$ for $n>n'$. This means
$\{x_1,...,x_k\}\subset U_n$ for $n>n'$. Consequently
$\mathcal{U}$ is an $\gamma$-cover of $X$.

\end{proof}

\begin{theorem}\label{th1} For a space $X$, the following statements are
equivalent:

\begin{enumerate}

\item $C_p(X)$ $\models$ $S_{1}(\mathcal{D},\mathcal{S})$;

\item $C_p(X)$ is strongly sequentially dense in itself and is
separable;

 \item $X$ $\models$ $S_{1}(\Omega, \Gamma)$ and
$iw(X)=\aleph_0$;

\item $C_p(X)$ $\models$ $S_{1}(\Omega_x,\Gamma_x)$ and is
separable;

\item $C_p(X)$ $\models$ $S_{1}(\mathcal{D},\Gamma_x)$ and is
separable;

\item $C_p(X)$ $\models$ $S_{fin}(\mathcal{D},\mathcal{S})$;

\item $X$ $\models$ $S_{fin}(\Omega, \Gamma)$ and
$iw(X)=\aleph_0$;

\item $C_p(X)$ $\models$ $S_{fin}(\Omega_x, \Gamma_x)$ and is
separable;

\item $C_p(X)$ $\models$ $S_{fin}(\mathcal{D}, \Gamma_x)$ and is
separable.

\end{enumerate}

\end{theorem}

\begin{proof}
$(1)\Rightarrow(2)$. Let $D$ be a dense subset of $C_p(X)$. By
$S_{1}(\mathcal{D},\mathcal{S})$, for
 sequence $\{D_i : D_i=D$ and $i\in \omega \}$
there is a sequence $(d_{i}: i\in\omega)$ such that for each $i$,
$d_{i}\in D_{i}$, and $\{d_{i}: i\in\omega \}$ is a countable
sequentially dense subset of $C_p(X)$. It follows that $D$ is a
sequentially dense subset of $C_p(X)$.

 $(2)\Rightarrow(1)$. Let $\{D_i\}_{i\in \omega}$ be a sequence
 of dense subsets of $C_{p}(X)$. Since $X$ $\models$ $S_{1}(\Omega,
 \Gamma)$, then,  $X$ $\models$ $S_{1}(\Omega, \Omega)$ and, by
 Theorem \ref{th11}, $C_p(X)$ $\models$ $S_{1}(\mathcal{D},\mathcal{D})$. Then there is a sequence $\{d_{i}\}_{i\in\omega}$ such that for each $i$,
$d_{i}\in D_{i}$, and $\{d_{i}: i\in\omega \}$ is a countable
dense  subset of $C_p(X)$. By $(2)$, $\{d_{i}: i\in\omega \}$ is a
countable sequentially dense subset of $C_p(X)$, i.e. $\{d_{i}:
i\in\omega \}\in \mathcal{S}$.

$(2)\Rightarrow(3)$.  By Theorem \ref{th71} and Theorem
\ref{th31}.

$(3)\Leftrightarrow(4)$. By Theorem \ref{th7}.

$(4)\Rightarrow(5)$ is immediate.

$(5)\Rightarrow(2)$. Let $D\in \mathcal{D}$, $f\in C(X)$ and
$\{D_n\}_{n\in \omega}$ such that $D_n=D$ for each $n\in \omega$.
By $(5)$, there is a sequence $\{f_{n}\}_{n\in\omega}$ such that
for each $n$, $f_{n}\in D_{n}$, and $\{f_{n}\}_{n\in\omega}$
converge to $f$. It follows that $D$ is a sequentially dense
subset of $C_p(X)$.

$(7)\Leftrightarrow(8)$. By Theorem \ref{th44} and Theorem
\ref{th31}.

$(8)\Rightarrow(9)$ is immediate.

$(1)\Rightarrow(6)$ is immediate.

$(3)\Rightarrow(7)$ is immediate.

$(9)\Rightarrow(2)$ is proved similarly the implication
$(5)\Rightarrow(2)$.

$(6)\Rightarrow(2)$ is proved similarly the implication
$(1)\Rightarrow(2)$.

\end{proof}

\begin{corollary} For a separable metrizable space $X$, the
following are equivalent:

\begin{enumerate}

\item $C_p(X)$ $\models$ $S_{1}(\mathcal{D},\mathcal{S})$;

\item $C_p(X)$ is strongly sequentially dense in itself;

\item $C_p(X)$ is strongly sequentially separable;

\item $C_p(X)$ $\models$ $S_{1}(\Omega_x,\Gamma_x)$;

\item $C_p(X)$ $\models$ $S_{1}(\mathcal{D},\Gamma_x)$;

\item $X$ $\models$ $S_{1}(\Omega, \Gamma)$;

\item $C_p(X)$ $\models$ $S_{fin}(\mathcal{D},\mathcal{S})$;

\item $C_p(X)$ $\models$ $S_{fin}(\Omega_x,\Gamma_x)$;

\item $C_p(X)$ $\models$ $S_{fin}(\mathcal{D},\Gamma_x)$;

\item $X$ $\models$ $S_{fin}(\Omega, \Gamma)$.

\end{enumerate}

\end{corollary}

\section{$S_1(\mathcal{S},\mathcal{D})$}

In \cite{sak1} (Theorem 2.3),  M. Sakai proved:

\begin{theorem}(Sakai) \label{th66} For a   space $X$, the following
statements are equivalent:

\begin{enumerate}

\item $C_p(X)$ $\models$ $S_{1}(\Gamma_x,\Omega_x)$ (the weak
sequence selection property);

\item $X$ $\models$ $S_{1}(\Gamma_{cl}, \Omega_{cl})$ and is
strongly zero-dimensional.

\end{enumerate}

\end{theorem}

\begin{definition}(Sakai)
An $\gamma$-cover $\mathcal{U}$ of co-zero sets of $X$  is {\bf
$\gamma_F$-shrinkable} if there exists a $\gamma$-cover $\{F(U) :
U\in \mathcal{U}\}$ of zero-sets of $X$ with $F(U)\subset U$ for
every $U\in \mathcal{U}$.
\end{definition}

For a topological space $X$ we denote:

$\bullet$ $\Gamma_F$ --- the family of $\gamma_F$-shrinkable
$\gamma$-covers of $X$.

\begin{proposition}\label{th86} For a strongly
zero-dimensional   space $X$, the following statements are
equivalent:

\begin{enumerate}

\item $X$ $\models$ $S_{1}(\Gamma_F, \Omega)$;

\item $X$ $\models$ $S_{1}(\Gamma_{cl}, \Omega_{cl})$.

\end{enumerate}

\end{proposition}

\begin{proposition} \label{pr23} For a   space $X$, the following statements are
equivalent:

\begin{enumerate}

\item $C_p(X)$ $\models$ $S_{1}(\Gamma_x,\Omega_x)$;

\item $X$ $\models$ $S_{1}(\Gamma_F, \Omega)$.

\end{enumerate}

\end{proposition}

\begin{proof} $(1)\Rightarrow(2)$. By Theorem \ref{th66} and
Proposition \ref{th86}.

 $(2)\Rightarrow(1)$. Let $X$ $\models$ $S_{1}(\Gamma_F, \Omega)$
and $\{A_{i}\}_{i\in \omega}$ such that $A_i\in \Gamma_{\bf{0}}$
for each $i\in \omega$. Consider
$\mathcal{U}_i=\{f^{-1}(-\frac{1}{i}, \frac{1}{i}): f\in A_i \}$
for each $i\in \omega$. Without loss of generality we can assume
that there is $i'$ that a set $U\neq X$ for any $i>i'$ and $U\in
\mathcal{U}_i$. Otherwise there is sequence $\{f_{i_k}\}_{k\in
\omega}$ such that $\{f_{i_k}\}_{k\in \omega}$ uniform converge to
$\bf{0}$ and $\{f_{i_k} : k\in \omega\}\in \Omega_{\bf 0}$.

Note that $\mathcal{F}_i=\{f^{-1}[-\frac{1}{i+1}, \frac{1}{i+1}]:
f\in A_i \}$ is $\gamma$-cover of zero-sets of $X$. It follows
that $\mathcal{U}_i\in \Gamma_F$ for each $i\in \omega$. By $X$
$\models$ $S_{1}(\Gamma_F, \Omega)$, there is a set $\{U_{i}:
i\in\omega\}$ such that for each $i$, $U_{i}\in \mathcal{U}_i$,
and $\{U_{i}: i\in\omega \}$ is an element of $\Omega$.

We claim that $\bf 0$ $\in \overline{\{f_{i} : i\in \omega \}}$.
Let $W=<$ $\bf 0$ $, K, \epsilon>$ be a base neighborhood of $\bf
0$ where $\epsilon>0$ and $K$ is a finite subset of $X$, then
there are $i_0\in \omega$ such that $\frac{1}{i_0}<\epsilon$ and
$U_{i_0}\supset K$. It follows that $f_{i_0}\in W$ and, hence,
$\bf 0$ $\in \overline{\{f_{i} : i\in \omega \}}$ and $C_p(X)$
$\models$ $S_{1}(\Gamma_x, \Omega_x)$.

 By Theorem
\ref{th66}, we have that $X$ is strongly zero-dimensional.

\end{proof}

\begin{theorem}\label{th50} For a   space $X$, the following
statements are equivalent:

\begin{enumerate}

\item $C_p(X)$ $\models$ $S_{1}(\mathcal{S},\mathcal{D})$ and is
sequentially separable;

\item $X$ $\models$ $S_{1}(\Gamma_F, \Omega)$,  $X$ $\models$ $V$;

\item $X$ $\models$ $S_{1}(\Gamma_{cl}, \Omega_{cl})$, $X$
$\models$ $V$ and is strongly zero-dimensional;

\item $C_p(X)$ $\models$ $S_{1}(\Gamma_x, \Omega_x)$ and is
sequentially separable;

\item $C_p(X)$ $\models$ $S_{1}(\mathcal{S}, \Omega_x)$ and is
sequentially separable.

\end{enumerate}

\end{theorem}

\begin{proof} $(1)\Rightarrow(2)$. Let $\{\mathcal{V}_i\}\subset \Gamma_F$ and
$\mathcal{S}=\{h_m : m\in \omega\}$ be a countable sequentially
dense subset of $C_p(X)$. For each $i\in \omega$ we consider a
countable sequentially dense subset $\mathcal{S}_i$ of $C_p(X)$
and $\mathcal{U}_i=\{ U^{m}_i\}_{m\in \omega}$ where
$\mathcal{U}_i\subset \mathcal{V}_i$ and $\mathcal{S}_i =\{
f^m_i\in C(X) : f^m_i\upharpoonright F(U^{m}_i)=h_m$ and
$f^m_i\upharpoonright (X\setminus U^{m}_i)=1$ for $m \in \omega
\}$. Note that $\mathcal{U}_i\in \Gamma_F$ for each $i\in \omega$.

Since $\mathcal{F}_i=\{F(U^{m}_i) : m\in \omega\}$ is infinite, it
is a $\gamma$-cover of zero subsets of $X$. Since $\mathcal{S}$ is
a countable sequentially dense subset of $C_p(X)$, we have that
$\mathcal{S}_i$ is a  countable sequentially dense subset of
$C_p(X)$ for each $i\in \omega$.  Let $h\in C(X)$, there is a
sequence $\{h_{m_s}: s\in \omega\}\subset \mathcal{S}$ such that
$\{h_{m_s}\}_{s\in \omega}$ converge to $h$.
 Let $K$ be a finite subset of $X$, $\epsilon>0$ and $W=<h, K,\epsilon>$ be
a base neighborhood of $h$, then there is a number $m_0$ such that
$K\subset F(U^{m}_i)$ for $m>m_0$ and $h_{m_s}\in W$ for
$m_s>m_0$. Since $f^{m_s}_i\upharpoonright K=
h_{m_s}\upharpoonright K$ for each $m_s>m_0$, $f^{m_s}_i\in W$ for
each $m_s>m_0$. It follows that a sequence $\{f^{m_s}_i\}_{s\in
\omega}$ converge to $h$.

By $C(X)\in S_{1}(\mathcal{S},\mathcal{D})$, there is a sequence
$\{f^{m(i)}_{i}: i\in\omega\}$ such that for each $i$,
$f^{m(i)}_{i}\in \mathcal{S}_i$, and $\{f^{m(i)}_{i}: i\in\omega
\}$ is an element of $\mathcal{D}$.

Consider a set $\{U^{m(i)}_{i}: i\in \omega\}$.

(a). $U^{m(i)}_{i}\in \mathcal{U}_{i}$.

(b). $\{U^{m(i)}_{i}: i\in \omega\}$ is a $\omega$-cover of $X$.

Let $K$ be a finite subset of $X$ and $U=<$ $\bf{0}$ $, K,
(-1,1)>$ be a base neighborhood of $\bf{0}$, then there is
$f^{m(i)_{j_0}}_{i_{j_0}}\in U$ for some $j_0\in \omega$. It
follows that $K\subset U^{m(i)_{j_0}}_{i_{j_0}}$. We thus get $X$
$\models$ $S_{1}(\Gamma_F, \Omega)$.

$(2)\Rightarrow(4)$. By Proposition \ref{pr23} and Theorem
\ref{th38}.




$(2)\Leftrightarrow(3)$. Clearly, that $\Gamma_{cl}\subset
\Gamma_F$. For a strongly zero-dimensional $X$, if $\mathcal{U}\in
\Gamma_F$, then there is $\mathcal{W}\in \Gamma_{cl}$ such that
$F(U)\subset W\subset U$ for every $U\in \mathcal{U}$ and $W\in
\mathcal{W}$.

$(3)\Leftrightarrow(4)$. By Theorem \ref{th38} and Theorem
\ref{th66}.

$(4)\Rightarrow(5)$ is immediate.

$(5)\Rightarrow(1)$. Suppose that $C_p(X)$ is sequentially
separable and $C_p(X)$ $\models$ $S_{1}(\mathcal{S}, \Omega_x)$.

Let $D=\{d_n: n\in \omega \}$ be a dense subspace of $C_p(X)$.
Given a sequence of sequentially dense subspace of $C_p(X)$,
enumerate it as $\{S_{n,m}: n,m \in \omega \}$. For each $n\in
\omega$, pick $d_{n,m}\in S_{n,m}$ so that $d_n\in
\overline{\{d_{n,m}: m\in \omega\}}$. Then $\{d_{n,m}: m,n\in
\omega\}$ is dense in $C_p(X)$.

\end{proof}

\begin{corollary}\label{th50} For a separable metrizable space $X$, the following
statements are equivalent:

\begin{enumerate}

\item $C_p(X)$ $\models$ $S_{1}(\mathcal{S},\mathcal{D})$;

\item $X$ $\models$ $S_{1}(\Gamma_F, \Omega)$;

\item $X$ $\models$ $S_{1}(\Gamma_{cl}, \Omega_{cl})$, and is
strongly zero-dimensional;

\item $C_p(X)$ $\models$ $S_{1}(\Gamma_x, \Omega_x)$;

\item $C_p(X)$ $\models$ $S_{1}(\mathcal{S}, \Omega_x)$.

\end{enumerate}

\end{corollary}

\section{$S_{fin}(\mathcal{S},\mathcal{D})$}

\begin{theorem}\label{th104} For a   space $X$, the following statements are
equivalent:

\begin{enumerate}

\item $C_p(X)$ $\models$ $S_{fin}(\Gamma_x,\Omega_x)$;

\item $X$ $\models$ $S_{fin}(\Gamma_F, \Omega)$.

\end{enumerate}

\end{theorem}

\begin{proof}

$(1)\Rightarrow(2)$ Let $\{\mathcal{V}_n\}_{n\in \omega}$ be a
sequence $\gamma_F$-shrinkable $\gamma$-covers of $X$. Let
$\mathcal{U}_n=\{U_{n,m}: m\in \omega\}\subset \mathcal{V}_n$ for
each $n\in \omega$. Note that $\mathcal{U}_n\in \Gamma_F$ for each
$n\in \omega$.

 For $n,m\in \omega$,
let $f_{n,m}:X \mapsto [0,1]$ be the continuous function
satisfying $F(U_{n,m})=f_{n,m}^{-1}(0)$ and $X\setminus
U_{n,m}=f_{n,m}^{-1}(1)$. For each $n\in \omega$, $\{F(U_{n,m}):
m\in \omega\}$ is a $\gamma$-cover of $X$, it follows that
$\{f_{n,m}\}_{m\in \omega}$ is a sequence converging pointwise to
$\bf{0}$. By $C_p(X)$ $\models$ $S_{fin}(\Gamma_x,\Omega_x)$,
there is a sequence
$\{F_n=\{f_{n,m_1},f_{n,m_2},...,f_{n,m_{k_n}}\}\}_{n\in \omega}$
such that $F_n\subset \{f_{n,m}\}_{m\in \omega}$ for each $n\in
\omega$ and $\bigcup_{n\in \omega} F_n \in \Omega_0$.
 Then $\bigcup_{n\in \omega} \{U_{n,m_1}, U_{n,m_2},...,U_{n,m_{k_n}} \}$ is an $\omega$-cover of $X$.

$(2)\Rightarrow(1)$. For each $n\in \omega$, let $A_n\in
\Gamma_{\bf{0}}$.

 For $n\in \omega$ and $f\in A_n$, let $Z_{n,f}=\{x\in X:
|f(x)|\leq\frac{1}{2^{n+1}}\}$, $U_{n,f}=\{x\in X:
|f(x)|<\frac{1}{2^{n}}\}$.  For each $n\in \omega$, we put
$\mathcal{U}_n=\{U_{n,f}: f\in A_n\}$. If the set $\{n\in \omega:
X\in \mathcal{U}_n \}$ is infinite, $X=U_{n_1,f_1}=U_{n_2,f_2}=...
$ for some sequences $\{n_j\}_{j\in \omega}$ and $f_i\in A_{n_i}$,
where $\{n_j\}_{j\in \omega}$ is strictly increasing. This means
that $\{f_{i}\}_{i\in \omega}$ is a sequence converging uniformly
to $\bf{0}$. If the set $\{n\in \omega: X\in \mathcal{U}_n \}$ is
finite, by removing such finitely many $n$'s we assume
$U_{n,f}\neq X$ for $n\in \omega$ and $f\in A_n$.

Note that each $\mathcal{U}_n$ is a $\gamma_F$-shrinkable
$\gamma$-covers of $X$. By $X$ $\models$ $S_{fin}(\Gamma_F,
\Omega)$, there is a sequence $\{U_{n,f_{n,1}}, ...,
U_{n,f_{n,k(n)}}\}_{n\in \omega}$ such that $U_{n,f_{n,i}}\in
\mathcal{U}_n$ for each $n\in \omega$, $i\in \overline{1,k(n)}$
and $\bigcup_{n\in \omega} \{U_{n,f_{n,1}},
U_{n,f_{n,2}},...,U_{n,f_{n,k(n)}} \}$ is an $\omega$-cover of
$X$. We claim a sequence
$\{F_n=\{f_{n,1},f_{n,2},...,f_{n,k(n)}\}\}_{n\in \omega}$ such
that $F_n\subset A_n$ for each $n\in \omega$ and $\bigcup_{n\in
\omega} F_n \in \Omega_0$.

Let $K$ be a finite subset of $X$ and let $\epsilon$ a positive
real number.

Because of $U_{n,f_{n,i}}\neq X$ for $n\in \omega$ and $i\in
\overline{1,k(n)}$, there are $n'\in \omega$ and  $i'\in
\overline{1,k(n')}$ such that $K\subset U_{n',f_{n',i'}}$ and
$\frac{1}{2^{n'}}<\epsilon$. Then $|f_{n',i'}(x)|<\epsilon$ for
any $x\in K$. Thus $C_p(X)$ $\models$
$S_{fin}(\Gamma_x,\Omega_x)$.

\end{proof}

\begin{theorem}\label{th105} For a   space $X$, the following statements are
equivalent:

\begin{enumerate}

\item $C_p(X)$ $\models$ $S_{fin}(\mathcal{S},\mathcal{D})$ and is
sequentially separable;

\item $X$ $\models$ $S_{fin}(\Gamma_F, \Omega)$, $X$ $\models$
$V$;

\item $C_p(X)$ $\models$ $S_{fin}(\Gamma_x, \Omega_x)$ and is
sequentially separable;

\item $C_p(X)$ $\models$ $S_{fin}(\mathcal{S}, \Omega_x)$ and is
sequentially separable.

\end{enumerate}

\end{theorem}

\begin{proof} $(1)\Rightarrow(2)$.
 Let $\{\mathcal{V}_i\}\subset \Gamma_F$ and
$\mathcal{S}=\{h_m : m\in \omega\}$ be a countable sequentially
dense subset of $C_p(X)$. For each $i\in \omega$ we consider a
countable sequentially dense subset $\mathcal{S}_i$ of $C_p(X)$
and $\mathcal{U}_i=\{ U^{m}_i : m\in \omega\}$ where
$\mathcal{U}_i\subset \mathcal{V}_i$ and

$\mathcal{S}_i =\{ f^m_i\in C(X) : f^m_i\upharpoonright
F(U^{m}_i)=h_m$ and $f^m_i\upharpoonright (X\setminus U^{m}_i)=1$
for $m \in \omega \}$. Note that $\mathcal{U}_i\in \Gamma_F$ for
each $i\in \omega$.

Since $\mathcal{F}_i=\{F(U^{m}_i): m\in \omega\}$ is infinite, it
is a $\gamma$-cover of zero subsets of $X$. Since $\mathcal{S}$ is
a countable sequentially dense subset of $C_p(X)$, we have that
$\mathcal{S}_i$ is a  countable sequentially dense subset of
$C_p(X)$ for each $i\in \omega$.  Let $h\in C(X)$, there is a
sequence $\{h_{m_s}\}_{s\in \omega}\subset \mathcal{S}$ such that
$\{h_{m_s}\}_{s\in \omega}$ converge to $h$.
 Let $K$
be a finite subset of $X$, $\epsilon>0$ and $W=<h, K,\epsilon>$ be
a base neighborhood of $h$, then there is a number $m_0$ such that
$K\subset F(U^{m}_i)$ for $m>m_0$ and $h_{m_s}\in W$ for
$m_s>m_0$. Since $f^{m_s}_i\upharpoonright K=
h_{m_s}\upharpoonright K$ for each $m_s>m_0$, $f^{m_s}_i\in W$ for
each $m_s>m_0$. It follows that a sequence $\{f^{m_s}_i\}_{s\in
\omega}$ converge to $h$.

By $C(X)\in S_{fin}(\mathcal{S},\mathcal{D})$, there is a sequence
$\{F_i=\{f_{i,m_1},f_{i,m_2},...,f_{i,m_{k_i}}\}\}_{i\in \omega}$
such that $F_i\subset \mathcal{S}_i$ for each $i\in \omega$ and
$\bigcup_{i\in \omega} F_i \in \mathcal{D}$.
 Then $\bigcup_{i\in \omega} \{U_{i,m_1}, U_{i,m_2},...,U_{i,m_{k_i}} \}$ is an $\omega$-cover of $X$.

Let $K$ be a finite subset of $X$ and $U=<$ $\bf{0}$ $, K,(-1,1)>$
be a base neighborhood of $\bf{0}$, then there is $f_{i',m(i')}\in
\bigcup_{i\in \omega} F_i$ for some $i'\in \omega$ such that
$f_{i',m(i')}\in U$. It follows that $K\subset U^{m(i')}_{i'}$. We
thus get $X$ $\models$ $S_{1}(\Gamma_F, \Omega)$.

$(2)\Rightarrow(3)$. By Theorem \ref{th104} and Theorem
\ref{th38}.

$(3)\Rightarrow(4)$ is immediate.

$(4)\Rightarrow(1)$. Suppose that $C_p(X)$ is sequentially
separable and $C_p(X)$ $\models$ $S_{fin}(\mathcal{S}, \Omega_x)$.

Let $D=\{d_n: n\in \omega \}$ be a dense subspace of $C_p(X)$.
Given a sequence of sequentially dense subspace of $C_p(X)$,
enumerate it as $\{S_{n,m}: n,m \in \omega \}$. For each $n\in
\omega$, pick $D_{n,m}=\{d_{n,m(1)},..., d_{n,m(n)}\} \subset
S_{n,m}$ so that $d_n\in \overline{\bigcup_{m\in \omega}
D_{n,m}}$. Then $\bigcup_{n,m\in \omega} D_{n,m}$ is dense in
$C_p(X)$.

\end{proof}

\begin{corollary}\label{th107} For a separable metrizable space $X$, the following statements are
equivalent:

\begin{enumerate}

\item $C_p(X)$ $\models$ $S_{fin}(\mathcal{S},\mathcal{D})$;

\item $X$ $\models$ $S_{fin}(\Gamma_F, \Omega)$;

\item $C_p(X)$ $\models$ $S_{fin}(\Gamma_x, \Omega_x)$;

\item $C_p(X)$ $\models$ $S_{fin}(\mathcal{S}, \Omega_x)$.

\end{enumerate}

\end{corollary}

\section{$S_{1}(\mathcal{S},\mathcal{S})$}

In \cite{sak1} (Theorem 2.5),  M. Sakai proved:

\begin{theorem} $(Sakai)$ \label{th151} For a    space $X$, the following statements are
equivalent:

\begin{enumerate}

\item  $C_p(X)$ $\models$ $S_{1}(\Gamma_x, \Gamma_x)$;

\item $X$ $\models$ $S_{1}(\Gamma_{cl}, \Gamma_{cl})$ and is
strongly zero-dimensional.

\end{enumerate}

\end{theorem}

\begin{proposition}\label{pr152} For a strongly
zero-dimensional   space $X$, the following statements are
equivalent:

\begin{enumerate}

\item $X$ $\models$ $S_{1}(\Gamma_F, \Gamma)$;

\item $X$ $\models$ $S_{1}(\Gamma_{cl}, \Gamma_{cl})$.

\end{enumerate}

\end{proposition}

\begin{proposition} \label{pr28} For a space $X$, the following statements are
equivalent:

\begin{enumerate}

\item $C_p(X)$ $\models$ $S_{1}(\Gamma_x,\Gamma_x)$;

\item $X$ $\models$ $S_{1}(\Gamma_F, \Gamma)$.

\end{enumerate}

\end{proposition}

\begin{proof} $(1)\Rightarrow(2)$. By Theorem \ref{th151} and
Proposition \ref{pr152}.

 $(2)\Rightarrow(1)$. Let $X$ $\models$ $S_{1}(\Gamma_F, \Gamma)$
and $\{A_{i}\}_{i\in \omega}$ such that $A_i\in \Gamma_{\bf{0}}$
for each $i\in \omega$. Consider
$\mathcal{U}_i=\{f^{-1}(-\frac{1}{i}, \frac{1}{i}): f\in A_i \}$
for each $i\in \omega$. Without loss of generality we can assume
that there is $i'$ that a set $U\neq X$ for any $i>i'$ and $U\in
\mathcal{U}_i$. Otherwise there is sequence $\{f_{i_k}\}_{k\in
\omega}$ such that $\{f_{i_k}\}_{k\in \omega}$ uniform converge to
$\bf{0}$ and $\{f_{i_k}: k\in \omega\}\in \Omega_{\bf 0}$.

Note that $\mathcal{F}_i=\{f^{-1}[-\frac{1}{i+1}, \frac{1}{i+1}]:
f\in A_i \}$ is $\gamma$-cover of zero-sets of $X$. It follows
that $\mathcal{U}_i\in \Gamma_F$ for each $i\in \omega$. By $X$
$\models$ $S_{1}(\Gamma_F, \Gamma)$, there is a set $\{U_{i}:
i\in\omega\}$ such that for each $i$, $U_{i}\in \mathcal{U}_i$,
and $\{U_{i}: i\in\omega \}$ is an element of $\Gamma$.

We claim that $\{f_{i} : i\in \omega \}\in \Gamma_{\bf{0}}$. Let
$W=<$ $\bf 0$ $, K, \epsilon>$ be a base neighborhood of $\bf 0$
where $\epsilon>0$ and $K$ is a finite subset of $X$, then there
are $i_0\in \omega$ such that $\frac{1}{i_0}<\epsilon$ and
$U_{i}\supset K$ for any $i>i_0$. It follows that $f_{i}\in W$ for
$i>i_0$, and $C_p(X)$ $\models$ $S_{1}(\Gamma_x, \Gamma_x)$.

 By Theorem
\ref{th66}, we have that $X$ is strongly zero-dimensional.

\end{proof}

\begin{theorem}\label{th158} For a   space $X$, the following statements are
equivalent:

\begin{enumerate}

\item $C_p(X)$ $\models$ $S_{1}(\mathcal{S},\mathcal{S})$;

\item $X$ $\models$ $S_{1}(\Gamma_F, \Gamma)$,  $X$ $\models$ $V$
and is strongly zero-dimensional;

\item $X$ $\models$ $S_{1}(\Gamma_{cl}, \Gamma_{cl})$, $X$
$\models$ $V$ and is strongly zero-dimensional;

\item  $C_p(X)$ $\models$ $S_{1}(\Gamma_x, \Gamma_x)$ and is
sequentially separable;

\item  $C_p(X)$ $\models$ $S_{1}(\mathcal{S}, \Gamma_x)$ and is
sequentially separable.

\end{enumerate}

\end{theorem}

\begin{proof} $(1)\Rightarrow(2)$. Let $\{\mathcal{V}_i\}\subset \Gamma_F$ and
$\mathcal{S}=\{h_m: m\in \omega\}$ be a countable sequentially
dense subset of $C_p(X)$. For each $i\in \omega$ we consider a
countable sequentially dense subset $\mathcal{S}_i$ of $C_p(X)$
and $\mathcal{U}_i=\{ U^{m}_i :m\in \omega\}\subset \mathcal{V}_i$
where

$\mathcal{S}_i=\{ f^m_i\in C(X) : f^m_i\upharpoonright
F(U^{m}_i)=h_m$ and $f^m_i\upharpoonright (X\setminus U^{m}_i)=1$
for $m \in \omega \}$. Note that $\mathcal{U}_i\in \Gamma_F$ for
each $i\in \omega$.

Since $\mathcal{F}_i=\{F(U^{m}_i): m\in \omega\}$ is infinity and
it is a $\gamma$-cover of zero-sets of $X$. Since $\mathcal{S}$ is
a countable sequentially dense subset of $C_p(X)$, we have that
$\mathcal{S}_i$ is a  countable sequentially dense subset of
$C_p(X)$ for each $i\in \omega$.  Let $h\in C(X)$, there is a set
$\{h_{m_s}: s\in \omega\}\subset \mathcal{S}$ such that
$\{h_{m_s}\}_{s\in \omega}$ converge to $h$.
 Let $K$
be a finite subset of $X$, $\epsilon>0$ and $W=<h, K,\epsilon>$ be
a base neighborhood of $h$, then there is a number $m_0$ such that
$K\subset F(U^{m}_i)$ for $m>m_0$ and $h_{m_s}\in W$ for
$m_s>m_0$. Since $f^{m_s}_i\upharpoonright K=
h_{m_s}\upharpoonright K$ for each $m_s>m_0$, $f^{m_s}_i\in W$ for
each $m_s>m_0$. It follows that a sequence $\{f^{m_s}_i\}_{s\in
\omega}$ converge to $h$.

Since  $C(X)$ $\models$ $S_{1}(\mathcal{S},\mathcal{S})$, there is
a sequence $\{f^{m(i)}_{i}\}_{i\in\omega}$ such that for each $i$,
$f^{m(i)}_{i}\in \mathcal{S}_i$, and $\{f^{m(i)}_{i}: i\in\omega
\}$ is an element of $\mathcal{S}$.

Consider a set $\{U^{m(i)}_{i}: i\in \omega\}$.

(a). $U^{m(i)}_{i}\in \mathcal{U}_{i}$.

(b). $\{U^{m(i)}_{i}: i\in \omega\}$ is a $\gamma$-cover of $X$.

There is a sequence $\{f^{m(i)_{j}}_{i_{j}}\}$ converge to
$\bf{0}$. Let $K$ be a finite subset of $X$ and $U=<$ $\bf{0}$ $,
K, (-1,1)>$ be a base neighborhood of $\bf{0}$, then there exists
$j_0\in \omega$ such that  $f^{m(i)_{j}}_{i_{j}}\in U$ for each
$j>j_0$. It follows that $K\subset U^{m(i)_{j}}_{i_{j}}$ for
$j>j_0$. We thus get $X$ $\models$ $S_{fin}(\Gamma_F, \Gamma)$.
But $S_{fin}(\Gamma_F, \Gamma)=S_{1}(\Gamma_F, \Gamma)$.

By Proposition \ref{pr23}, $X$ $\models$ $S_{1}(\Gamma_F, \Omega)$
implies $C_p(X)$ $\models$ $S_{1}(\Gamma_x, \Omega_x)$. By Theorem
\ref{th66}, $X$ is strongly zero-dimensional.

$(2)\Rightarrow(1)$.  Fix $\{S_i: i\in \omega\}\subset
\mathcal{S}$ and $S=\{h_i: i\in \omega\}\in \mathcal{S}$. For each
$i\in \omega$ we consider a set $\{f^{i}_k: k\in \omega\}\subset
S_i$ such that $\{f^{i}_k\}_{k\in \omega}$ converge to $h_i$. For
each $i,k\in \omega$, we put $U_{i,k}=\{x\in X :
|f^{i}_k(x)-h_i(x)|<\frac{1}{i}\}$, $Z_{i,k}=\{x\in X :
|f^{i}_k(x)-h_i(x)|\leq\frac{1}{i+1}\}$. Each $U_{i,k}$ (resp.,
$Z_{i,k}$) is a cozero-set (resp., zero-set) in $X$ with
$Z_{i,k}\subset U_{i,k}$. Let $\mathcal{U}_i=\{ U_{i,k} : k\in
\omega\}$ and $\mathcal{Z}_i=\{ Z_{i,k} : k\in \omega\}$. So
without loss of generality, we may assume $U_{i,k}\neq X$ for each
$i,k\in \omega$. We can easily check that the condition $f^{i}_k
\rightarrow h_i$ ($k \rightarrow \infty $) implies that
$\mathcal{Z}_i$ is a $\gamma$-cover of $X$. Since  $X$ $\models$
$S_{1}(\Gamma_F, \Gamma)$ there is  a sequence $\{U_{i,k(i)}\}_{
i\in\omega}$ such that for each $i$, $U_{i,k(i)}\in
\mathcal{U}_i$, and $\{U_{i,k(i)}: i\in\omega \}$ is an element of
$\Gamma$. We claim that $\{f^{i}_{k(i)}\}_{i\in \omega}\in
\mathcal{S}$. For  $f\in C(X)$ there is a set $\{h_{i_s}: s\in
\omega\}\subset S$ such that $\{h_{i_s}\}_{s\in \omega}$ converge
to $f$. Then a sequence $\{f^{i_s}_{k(i_s)}\}_{s\in \omega}$
converge to $f$ too. Let $K$ be a finite subset of $X$,
$\epsilon>0$, and $U=< f, K, \epsilon>$ be a base neighborhood of
$f$, then there exists $m\in \omega$ such that  $h_{i_s}\in < f,
K, \frac{\epsilon}{2}>$ for each $i_s>m$. Since $\{U_{i,k(i)}:
i\in\omega \}$ is an element of $\Gamma$, there exists $n>m$ such
that $\frac{1}{n}<\frac{\epsilon}{2}$ and $K\subset U_{i,k(i)}$
for $i>n$. It follows that for each $i_s>n$ and $x\in K$ we have
that $|f(x)-f^{i_s}_{k(i_s)}(x)|\leq
|f(x)-h_{i_s}(x)|+|f^{i_s}_{k(i_s)}(x)-h_{i_s}(x)|<\frac{\epsilon}{2}+\frac{\epsilon}{2}=\epsilon$.
Hence a sequence $\{f^{i_s}_{k(i_s)}\}_{s\in \omega}$ converge to
$f$ and $\{f^{i}_{k(i)}\}_{i\in \omega}\in \mathcal{S}$.

$(2)\Leftrightarrow(3)\Leftrightarrow(4)$. By Theorem \ref{th151},
Proposition \ref{pr152}, Proposition \ref{pr28}  and Theorem
\ref{th38}.

$(4)\Rightarrow(5)$ is immediate.

$(5)\Rightarrow(2)$. Let $C_p(X)$ $\models$ $S_{1}(\mathcal{S},
\Gamma_x)$ and $C_p(X)$ be a sequentially separable.

Let $\{\mathcal{V}_i\}\subset \Gamma_F$ and $\mathcal{S}=\{h_m:
m\in \omega\}$ be a countable sequentially dense subset of
$C_p(X)$. For each $i\in \omega$ we consider a countable
sequentially dense subset $\mathcal{S}_i$ of $C_p(X)$ and
$\mathcal{U}_i=\{ U^{m}_i\}_{m\in \omega}\subset \mathcal{V}_i $
where

$\mathcal{S}_i=\{ f^m_i\in C(X) : f^m_i\upharpoonright
F(U^{m}_i)=h_m$ and $f^m_i\upharpoonright (X\setminus U^{m}_i)=1$
for $m \in \omega \}$.

Since $\mathcal{F}_i=\{F(U^{m}_i): m\in \omega\}$ is infinity, it
is a $\gamma$-cover of zero-subsets in $X$. Since $\mathcal{S}$ is
a countable sequentially dense subset of $C_p(X)$, we have that
$\mathcal{S}_i$ is a  countable sequentially dense subset of
$C_p(X)$ for each $i\in \omega$.

Let $h\in C(X)$, there is a set $\{h_{m_s}: s\in \omega\}\subset
\mathcal{S}$ such that $\{h_{m_s}\}_{s\in \omega}$ converge to
$h$.
 Let $K$ be a finite subset of $X$, $\epsilon>0$ and $W=<h, K,\epsilon>$ be
a base neighborhood of $h$, then there is a number $m_0$ such that
$K\subset F(U^{m}_i)$ for $m>m_0$ and $h_{m_s}\in W$ for
$m_s>m_0$. Since $f^{m_s}_i\upharpoonright K=
h_{m_s}\upharpoonright K$ for each $m_s>m_0$, $f^{m_s}_i\in W$ for
each $m_s>m_0$. It follows that a sequence $\{f^{m_s}_i\}_{s\in
\omega}$ converge to $h$.

By $C_p(X)$ $\models$ $S_{1}(\mathcal{S}, \Gamma_x)$, there is a
sequence $\{f^{m(i)}_{i}: i\in\omega\}$ such that for each $i$,
$f^{m(i)}_{i}\in \mathcal{S}_i$, and $\{f^{m(i)}_{i}: i\in\omega
\}$ is an element of $\Gamma_0$.

Consider a set $\{U^{m(i)}_{i}: i\in \omega\}$.

(a). $U^{m(i)}_{i}\in \mathcal{U}_{i}$.

(b). $\{U^{m(i)}_{i}: i\in \omega\}$ is a $\gamma$-cover of $X$.

Let $K$ be a finite subset of $X$ and $U=<$ $\bf{0}$ $, K,
\frac{1}{2}>$ be a base neighborhood of $\bf{0}$, then there is
$j_0\in \omega$ such that $f^{m(i)_{j}}_{i_{j}}\in U$ for each
$j>j_0$. It follows that $K\subset U^{m(i)_{j}}_{i_{j}}$ for each
$j>j_0$. We thus get $X$ $\models$ $S_{1}(\Gamma_F, \Gamma)$. By
Theorem \ref{th38}, $X$ $\models$ $V$. Since $C_p(X)$ $\models$
$S_{1}(\mathcal{S}, \Gamma_x)$  implies that $C_p(X)$ $\models$
$S_{1}(\mathcal{S}, \Omega_x)$, by Theorem \ref{th50}, we have
that $X$ is strongly zero-dimensional.

\end{proof}

\begin{corollary}\label{th159} For a separable metrizable space $X$, the following statements are
equivalent:

\begin{enumerate}

\item $C_p(X)$ $\models$ $S_{1}(\mathcal{S},\mathcal{S})$;

\item $X$ $\models$ $S_{1}(\Gamma_F, \Gamma)$ and is strongly
zero-dimensional;

\item $X$ $\models$ $S_{1}(\Gamma_{cl}, \Gamma_{cl})$ and is
strongly zero-dimensional;

\item  $C_p(X)$ $\models$ $S_{1}(\Gamma_x, \Gamma_x)$;

\item  $C_p(X)$ $\models$ $S_{1}(\mathcal{S}, \Gamma_x)$.

\end{enumerate}

\end{corollary}

The proof of fact that $S_{fin}(\Gamma_F, \Gamma)=S_{1}(\Gamma_F,
\Gamma)$  (or $S_{1}(\Gamma_{cl},
\Gamma_{cl})=S_{fin}(\Gamma_{cl}, \Gamma_{cl})$)  is analogous to
proof of Theorem 1.1 (in \cite{jmss}) that $S_{1}(\Gamma,
\Gamma)=S_{fin}(\Gamma,\Gamma)$.

\begin{proposition}\label{pr172} For a   space $X$, the following statements are
equivalent:

\begin{enumerate}

\item $C_p(X)$ $\models$ $S_{fin}(\Gamma_x, \Gamma_x)$;

\item $X$ $\models$ $S_{fin}(\Gamma_F, \Gamma)$;

\item $X$ $\models$ $S_{fin}(\Gamma_{cl}, \Gamma_{cl})$ and is
strongly zero-dimensional;

\item $X$ $\models$ $S_{1}(\Gamma_F, \Gamma)$.

\end{enumerate}

\end{proposition}

\begin{proof}

 Note that, by Theorem 2 in \cite{scheep}, $C_p(X)$ $\models$ $S_{1}(\Gamma_x, \Gamma_x)$ iff $C_p(X)$ $\models$ $S_{fin}(\Gamma_x, \Gamma_x)$.
By Theorem \ref{th151} and Proposition \ref{pr152}, we obtain the
complete proof.
\end{proof}

\begin{theorem}\label{th173} For a   space $X$, the following statements are
equivalent:

\begin{enumerate}

\item $C_p(X)$ $\models$ $S_{fin}(\mathcal{S},\mathcal{S})$ and is
sequentially separable;

\item $X$ $\models$ $S_{1}(\Gamma_F, \Gamma)$,  $X$ $\models$ $V$
and is strongly zero-dimensional;

\item $X$ $\models$ $S_{fin}(\Gamma_F, \Gamma)$,  $X$ $\models$
$V$ and is strongly zero-dimensional;

\item $X$ $\models$ $S_{fin}(\Gamma_{cl}, \Gamma_{cl})$, $X$
$\models$ $V$ and is strongly zero-dimensional;

\item  $C_p(X)$ $\models$ $S_{fin}(\Gamma_x, \Gamma_x)$ and is
sequentially separable;

\item  $C_p(X)$ $\models$ $S_{fin}(\mathcal{S}, \Gamma_x)$ and is
sequentially separable.

\end{enumerate}

\end{theorem}

\begin{proof} By Proposition \ref{pr172}, Theorem \ref{th158} and Theorem \ref{th38}.

\end{proof}

\begin{corollary}\label{th178} For a separable metrizable space $X$, the following statements are
equivalent:

\begin{enumerate}

\item $C_p(X)$ $\models$ $S_{fin}(\mathcal{S},\mathcal{S})$;

\item $X$ $\models$ $S_{1}(\Gamma_F, \Gamma)$, and is strongly
zero-dimensional;

\item $X$ $\models$ $S_{fin}(\Gamma_F, \Gamma)$,  and is strongly
zero-dimensional;

\item $X$ $\models$ $S_{fin}(\Gamma_{cl}, \Gamma_{cl})$, and is
strongly zero-dimensional;

\item  $C_p(X)$ $\models$ $S_{fin}(\Gamma_x, \Gamma_x)$;

\item  $C_p(X)$ $\models$ $S_{fin}(\mathcal{S}, \Gamma_x)$.

\end{enumerate}

\end{corollary}

\section{$U_{fin}(\mathcal{S},\mathcal{D})$ }

Recall that $U_{fin}(\mathcal{S},\mathcal{D})$ is the selection
hypothesis:
 whenever $\mathcal{U}_1$, $\mathcal{U}_2, ... \in \mathcal{S}$, there are finite sets $\mathcal{F}_n\subseteq \mathcal{U}_n$,
$n\in \omega$, such that $\{\bigcup \mathcal{F}_n : n\in
\omega\}\in \mathcal{D}$. For a function space $C(X)$, we can
represent the condition  $\{\bigcup \mathcal{F}_n : n\in
\omega\}\in \mathcal{D}$ as $\forall$ $f\in C(X)$ $\forall$ a base
neighborhood $O(f)=<f, K, \epsilon >$ of $f$  where $\epsilon>0$
and $K=\{x_1, ..., x_k\}$ is a finite subset of $X$, there is
$n'\in \omega$ such that for each $j\in \{1,...,k\}$ there is
$g\in \mathcal{F}_{n'}$ such that $g(x_j)\in (f(x_j)-\epsilon,
f(x_j)+\epsilon)$.

Similarly, $U_{fin}(\Gamma_0,\Omega_0)$: whenever $\mathcal{S}_1$,
$\mathcal{S}_2, ... \in \Gamma_0$, there are finite sets
$\mathcal{F}_n\subseteq \mathcal{S}_n$, $n\in \omega$, such that
$\{\bigcup \mathcal{F}_n : n\in \omega\}\in \Omega_0$, i.e. for a
base neighborhood $O(f)=<f, K, \epsilon
>$ of $f={\bf 0}$  where $\epsilon>0$
and $K=\{x_1, ..., x_k\}$ is a finite subset of $X$, there is
$n'\in \omega$ such that for each $j\in \{1,...,k\}$ there is
$g\in \mathcal{F}_{n'}$ such that $g(x_j)\in (f(x_j)-\epsilon,
f(x_j)+\epsilon)$.

\begin{theorem}\label{th198} For a   space $X$, the following statements are
equivalent:

\begin{enumerate}

\item $C_p(X)$ $\models$ $U_{fin}(\Gamma_x,\Omega_x)$;

\item $X$ $\models$ $U_{fin}(\Gamma_F, \Omega)$.

\end{enumerate}

\end{theorem}

\begin{proof}

 $(1)\Rightarrow(2)$. Let $\{\mathcal{V}_i\}\subset \Gamma_F$. For each $i\in \omega$ and $\mathcal{U}_i=\{U^m_i\}\subset \mathcal{V}_i$
 we consider $\mathcal{K}_i =\{ f^m_i\in
C(X) : f^m_i\upharpoonright F(U^{m}_i)=0$ and
$f^m_i\upharpoonright (X\setminus U^{m}_i)=1$ for $m \in \omega
\}$. Note that $\mathcal{U}_i\in \Gamma_F$ for each $i\in \omega$.

Since $\mathcal{F}_i=\{F(U^{m}_i): m\in \omega\}$ is a
$\gamma$-cover of zero-sets of $X$, we have that $\mathcal{K}_i$
converge to $\bf{0}$ for each $i\in \omega$. By $C_p(X)$ $\models$
$U_{fin}(\Gamma_x,\Omega_x)$, there are finite sets
$F_i=\{f^{m_1}_i, ..., f^{m_{s(i)}}_i\}\subseteq \mathcal{K}_i$
such that $\{\bigcup F_i : i\in \omega\}\in \Omega_{0}$. Note that
$\{\bigcup \{U^{m_1}_i, ..., U^{m_{s(i)}}_i\}  : i\in \omega\}\in
\Omega$.

 $(2)\Rightarrow(1)$. Let $X$ $\models$ $U_{fin}(\Gamma_F, \Omega)$
and $A_i\in \Gamma_{\bf{0}}$ for each $i\in \omega$. Consider
$\mathcal{U}_i=\{U_{i,f}=f^{-1}(-\frac{1}{i}, \frac{1}{i}): f\in
A_i \}$ for each $i\in \omega$. Without loss of generality we can
assume that a set $U_{i,f}\neq X$ for any $i\in \omega$ and $f\in
A_i$. Otherwise there is sequence $\{f_{i_k}\}_{k\in \omega}$ such
that $\{f_{i_k}\}_{k\in \omega}$ uniform converge to $\bf{0}$ and
$\{f_{i_k}: k\in \omega\}\in \Omega_{\bf 0}$.

Note that $\mathcal{F}_i=\{F_{i,m}\}_{m\in
\omega}=\{f^{-1}_{i,m}[-\frac{1}{i+1}, \frac{1}{i+1}]: m\in \omega
\}$ is $\gamma$-cover of zero-sets of $X$ and $F_{i,m}\subset
U_{i,m}$ for each $i,m\in \omega$. It follows that
$\mathcal{U}_i\in \Gamma_F$ for each $i\in \omega$.

By $X$ $\models$ $U_{fin}(\Gamma_F, \Omega)$, there is a sequence
$\{U_{i,m(1)}, U_{i,m(2)}, ..., U_{i,m(i)}: i\in\omega\}$ such
that for each $i$ and $k\in\{m(1),...,m(i)\}$, $U_{i,m(k)}\in
\mathcal{U}_i$, and

$\{\bigcup \{U_{i,m(1)}, ..., U_{i,m(i)}\}: i\in \omega\}\in
\Omega$.

We claim that $\{\bigcup \{f_{i,m(1)}, ..., f_{i,m(i)}\}: i\in
\omega\}\in \Omega_0$.

 Let $W=<$ $\bf 0$ $, K, \epsilon>$ be a base neighborhood of
$\bf 0$ where $\epsilon>0$ and $K=\{x_1, ..., x_s\}$ is a finite
subset of $X$, then there are $i_0, i_1\in \omega$ such that
$\frac{1}{i_0}<\epsilon$, $i_1>i_0$ and $\bigcup_{k=m(1)}^{m(i_1)}
U_{i_1,k}\supset K$. It follows that for each $j\in \{1,...,s\}$
there is $g\in \{f_{i_1,m(1)}, ..., f_{i_1,m(i_1)}\}$ such that
$g(x_j)\in (-\epsilon, \epsilon)$.

\end{proof}

\begin{theorem}\label{th195} For a   space $X$, the following statements are
equivalent:

\begin{enumerate}

\item $C_p(X)$ $\models$ $U_{fin}(\mathcal{S},\mathcal{D})$ and is
sequentially separable;

\item $X$ $\models$ $U_{fin}(\Gamma_F, \Omega)$, $X$ $\models$
$V$;

\item $C_p(X)$ $\models$ $U_{fin}(\Gamma_x, \Omega_x)$ and is
sequentially separable;

\item $C_p(X)$ $\models$ $U_{fin}(\mathcal{S}, \Omega_x)$ and is
sequentially separable.

\end{enumerate}

\end{theorem}

\begin{proof} $(1)\Rightarrow(2)$. Suppose that $C_p(X)$ $\models$ $U_{fin}(\mathcal{S},\mathcal{D})$ and is sequentially separable.
Let $\{\mathcal{V}_i\}\subset \Gamma_F$ and $\mathcal{S}=\{h_j:
j\in \omega\}$  be a countable sequentially dense subset of
$C_p(X)$.

For each $i\in \omega$ and $\mathcal{U}_i=\{U^j_i: j\in
\omega\}\subset \mathcal{V}_i$
 we consider $\mathcal{S}_i =\{ f^j_i\in
C(X) : f^j_i\upharpoonright F(U^{j}_i)=h_j$ and
$f^j_i\upharpoonright (X\setminus U^{j}_i)=1$ for $j \in \omega
\}$.

Since $\mathcal{F}_i=\{F(U^{m}_i): m\in \omega\}$ is a
$\gamma$-cover of $X$, we have that $\mathcal{S}_i$ is a countable
sequentially dense subset of $C_p(X)$ for each $i\in \omega$.

By $C_p(X)$ $\models$ $U_{fin}(\mathcal{S},\mathcal{D})$, there
are finite sets $F_i=\{f^{m_1}_i, ..., f^{m_{s(i)}}_i\}\subseteq
\mathcal{S}_i$ such that $\{\bigcup F_i : i\in \omega\}\in
\mathcal{D}$. Note that $\{\bigcup \{U^{m_1}_i, ...,
U^{m_{s(i)}}_i\} : i\in \omega\}\in \Omega$. By Theorem
\ref{th38}, $X$ $\models$ $V$.

$(2)\Rightarrow(3)$. By Theorem \ref{th38} and Theorem
\ref{th198}.

$(3)\Rightarrow(4)$ is immediate.

$(4)\Rightarrow(1)$. Suppose that $C_p(X)$ is sequentially
separable and $C_p(X)$ $\models$ $U_{fin}(\mathcal{S}, \Omega_x)$.

Let $D=\{d_n: n\in \omega \}$ be a dense subspace of $C_p(X)$.
Given a sequence of sequentially dense subspace of $C_p(X)$,
enumerate it as $\{S_{n,m}: n,m \in \omega \}$. For each $n\in
\omega$, pick

$\mathcal{F}_{n,m}=\{d_{n,m,1},..., d_{n,m,k(n,m)} \} \subset
S_{n,m}$ so that $d_n\in \overline{\{\bigcup \mathcal{F}_{n,m}:
m\in \omega\}}$, i.e. for a base neighborhood $O(d_n)=<d_n, K,
\epsilon
>$ of $d_n$  where $\epsilon>0$
and $K=\{x_1, ..., x_k\}$ is a finite subset of $X$, there is
$m'\in \omega$ such that for each $j\in \{1,...,k\}$ there is
$g\in \mathcal{F}_{n,m'}$ such that $g(x_j)\in (d_n(x_j)-\epsilon,
d_n(x_j)+\epsilon)$.

 Then $\{\bigcup \mathcal{F}_{n,m}: m,n\in
\omega\}\in \mathcal{D}$.

\end{proof}

\begin{theorem}\label{th1951} For a separable metrizable space $X$, the following statements are
equivalent:

\begin{enumerate}

\item $C_p(X)$ $\models$ $U_{fin}(\mathcal{S},\mathcal{D})$;

\item $X$ $\models$ $U_{fin}(\Gamma, \Omega)$;

\item $C_p(X)$ $\models$ $U_{fin}(\Gamma_x, \Omega_x)$;

\item $C_p(X)$ $\models$ $U_{fin}(\mathcal{S}, \Omega_x)$.

\end{enumerate}

\end{theorem}

\section{$U_{fin}(\mathcal{S},\mathcal{S})$}

Recall that $U_{fin}(\mathcal{S},\mathcal{S})$ is the selection
hypothesis:
 whenever $\mathcal{U}_1$, $\mathcal{U}_2, ... \in \mathcal{S}$, there are finite sets $\mathcal{F}_n\subseteq \mathcal{U}_n$,
$n\in \omega$, such that $\{\bigcup \mathcal{F}_n : n\in
\omega\}\in \mathcal{S}$. For a function space $C(X)$, we can
represent the condition  $\{\bigcup \mathcal{F}_n : n\in
\omega\}\in \mathcal{S}$ as $\forall$ $f\in C(X)$ $\forall$ a base
neighborhood of $f$ $O(f)=<f, K, \epsilon >$ of $f$  where
$\epsilon>0$ and $K=\{x_1, ..., x_k\}$ is a finite subset of $X$,
there is $n'\in \omega$ such that for each $n>n'$ and $j\in
\{1,...,k\}$ there is $g\in \mathcal{F}_{n}$ such that $g(x_j)\in
(f(x_j)-\epsilon, f(x_j)+\epsilon)$.

Similarly, $U_{fin}(\Gamma_0,\Gamma_0)$: whenever $\mathcal{S}_1$,
$\mathcal{S}_2, ... \in \Gamma_0$, there are finite sets
$\mathcal{F}_n\subseteq \mathcal{S}_n$, $n\in \omega$, such that
$\{\bigcup \mathcal{F}_n : n\in \omega\}\in \Gamma_0$, i.e. for a
base neighborhood $O(f)=<f, K, \epsilon
>$ of $f={\bf 0}$  where $\epsilon>0$
and $K=\{x_1, ..., x_k\}$ is a finite subset of $X$, there is
$n'\in \omega$ such that for each $n>n'$ and $j\in \{1,...,k\}$
there is $g\in \mathcal{F}_{n}$ such that $g(x_j)\in
(f(x_j)-\epsilon, f(x_j)+\epsilon)$.

\begin{theorem}\label{th194} For a   space $X$, the following statements are
equivalent:

\begin{enumerate}

\item $C_p(X)$ $\models$ $U_{fin}(\Gamma_x,\Gamma_x)$;

\item $X$ $\models$ $U_{fin}(\Gamma_{F}, \Gamma)$.

\end{enumerate}

\end{theorem}

\begin{proof} $(1)\Rightarrow(2)$.  Let $\{\mathcal{V}_i\}\subset \Gamma_F$. For each $i\in \omega$ we consider a
 subset $\mathcal{S}_i$ of $C_p(X)$ and $\mathcal{U}_i=\{
U^{m}_i\}_{m\in \omega}\subset \mathcal{V}_i$ where

$\mathcal{S}_i =\{ f^m_i\in C(X) : f^m_i\upharpoonright
F(U^{m}_i)=0$ and $f^m_i\upharpoonright (X\setminus U^{m}_i)=1$
for $m \in \omega \}$.

Since $\mathcal{F}_i=\{F(U^{m}_i): m\in \omega\}$ is a
$\gamma$-cover of $X$, we have that $\mathcal{S}_i$ converge to
${\bf 0}$, i.e. $\mathcal{S}_i\in \Gamma_0$ for each $i\in
\omega$.

Since  $C(X)$ $\models$ $U_{fin}(\Gamma_x,\Gamma_x)$, there is a
sequence $\{\mathcal{F}_i\}_{i\in \omega}=\{f^{m_1}_{i},...,
f^{m_{k(i)}}_{i} : i\in\omega\}$ such that for each $i$,
$\mathcal{F}_i\subseteq \mathcal{S}_i$, and $\{\bigcup
\mathcal{F}_i\}_{i\in \omega}\in \Gamma_0$.

Consider a sequence $\{W_i\}_{i\in \omega}=\{U^{m_1}_{i},
...,U^{m_{k(i)}}_{i} : i\in \omega\}$.

(a). $W_i \subset \mathcal{U}_{i}$.

(b). $\{\bigcup W_i: i\in \omega\}$ is a $\gamma$-cover of $X$.

 Let $K=\{x_1,...,x_s\}$ be a finite subset of $X$ and $U=<$ $\bf{0}$ $,
K, \frac{1}{2}>$ be a base neighborhood of $\bf{0}$, then there
exists $i_0\in \omega$ such that for each $i>i_0$ and

$j\in \{1,...,s\}$ there is $g\in \mathcal{F}_{i}$ such that
$g(x_j)\in (-\frac{1}{2}, \frac{1}{2})$.

 It follows that
$K\subset \bigcup\limits_{j=1}^{k(i)} U^{m_j}_{i}$ for $i>i_0$. We
thus get $X$ $\models$ $U_{fin}(\Gamma_{F}, \Gamma)$.

$(2)\Rightarrow(1)$.  Fix $\{S_i: i\in \omega\}\subset \Gamma_0$
where $S_i=\{f^{i}_k: k\in \omega\}$ for each $i\in \omega$.

For each $i,k\in \omega$, we put $U_{i,k}=\{x\in X :
|f^{i}_k(x)|<\frac{1}{i}\}$, $Z_{i,k}=\{x\in X :
|f^{i}_k(x)|\leq\frac{1}{i+1}\}$.

Each $U_{i,k}$ (resp., $Z_{i,k}$) is a cozero-set (resp.,
zero-set) in $X$ with $Z_{i,k}\subset U_{i,k}$. Let
$\mathcal{U}_i=\{ U_{i,k} : k\in \omega\}$ and $\mathcal{Z}_i=\{
Z_{i,k} : k\in \omega\}$. So without loss of generality, we may
assume $U_{i,k}\neq X$ for each $i,k\in \omega$. We can easily
check that the condition $f^{i}_k \rightarrow \bf{0}$ ($k
\rightarrow \infty $) implies that $\mathcal{Z}_i$ is a
$\gamma$-cover of $X$.

Since  $X$ $\models$ $U_{fin}(\Gamma_{F}, \Gamma)$ there is a
sequence $\{\mathcal{F}_i\}_{i\in
\omega}=\{U_{i,k_1},...,U_{i,k_i}: i\in\omega\}$ such that for
each $i$, $\mathcal{F}_i \subset \mathcal{U}_i$, and $\{\bigcup
\mathcal{F}_i: i\in\omega \}$ is an element of $\Gamma$.

Let $K=\{x_1,...,x_s\}$ be a finite subset of $X$, $\epsilon>0$,
and $U=<$ $\bf{0}$, $K, \epsilon>$ be a base neighborhood of
$\bf{0}$, then there exists $i'\in \omega$ such that for each
$i>i'$ $K\subset \bigcup \mathcal{F}_i$. It follow that for each
$i>i'$ and $j\in \{1,...,s\}$ there is $g\in \mathcal{S}_{i}$ such
that $g(x_j)\in (-\epsilon, \epsilon)$. So $C_p(X)$ $\models$
$U_{fin}(\Gamma_x,\Gamma_x)$.

\end{proof}

\begin{theorem}\label{th153} For a   space $X$, the following statements are
equivalent:

\begin{enumerate}

\item $C_p(X)$ $\models$ $U_{fin}(\mathcal{S},\mathcal{S})$ and is
sequentially separable;

\item $X$ $\models$ $U_{fin}(\Gamma_F, \Gamma)$,  $X$ $\models$
$V$;

\item  $C_p(X)$ $\models$ $U_{fin}(\Gamma_x, \Gamma_x)$ and is
sequentially separable;

\item  $C_p(X)$ $\models$ $U_{fin}(\mathcal{S}, \Gamma_x)$ and is
sequentially separable.

\end{enumerate}

\end{theorem}

\begin{proof} $(1)\Rightarrow(2)$.
  Let $\{\mathcal{V}_i\}\subset \Gamma_F$ and
$\mathcal{S}=\{h_m: m\in \omega\}$ be a countable sequentially
dense subset of $C_p(X)$. For each $i\in \omega$ we consider a
countable sequentially dense subset $\mathcal{S}_i$ of $C_p(X)$
and $\mathcal{U}_i=\{ U^{m}_i\}_{m\in \omega}\subset
\mathcal{V}_i$ where

$\mathcal{S}_i =\{ f^m_i\in C(X) : f^m_i\upharpoonright
F(U^{m}_i)=h_m$ and $f^m_i\upharpoonright (X\setminus U^{m}_i)=1$
for $m \in \omega \}$.

Since $\mathcal{F}_i=\{F(U^{m}_i): m\in \omega\}$ is a
$\gamma$-cover of zero-sets of $X$ and $\mathcal{S}$ is a
countable sequentially dense subset of $C_p(X)$, we have that
$\mathcal{S}_i$ is a  countable sequentially dense subset of
$C_p(X)$ for each $i\in \omega$.  Let $h\in C(X)$, there is a
sequence $\{h_{m_s}: s\in \omega\}\subset \mathcal{S}$ such that
$\{h_{m_s}\}_{s\in \omega}$ converge to $h$.
 Let $K$ be a finite subset of $X$, $\epsilon>0$ and $W=<h, K,\epsilon>$ be
a base neighborhood of $h$, then there is a number $m_0$ such that
$K\subset F(U^{m}_i)$ for $m>m_0$ and $h_{m_s}\in W$ for
$m_s>m_0$. Since $f^{m_s}_i\upharpoonright K=
h_{m_s}\upharpoonright K$ for each $m_s>m_0$, $f^{m_s}_i\in W$ for
each $m_s>m_0$. It follows that a sequence $\{f^{m_s}_i\}_{s\in
\omega}$ converge to $h$.

Since  $C(X)$ $\models$ $U_{fin}(\mathcal{S},\mathcal{S})$, there
is a sequence $\{F_i\}=\{f^{m_1}_{i}, ...,f^{m_s}_{i} :
i\in\omega\}$ such that for each $i$, $F_i \subset \mathcal{S}_i$,
and $\{\bigcup F_i: i\in\omega \}$ is an element of $\mathcal{S}$,
i.e. for any $f\in C(X)$ and a base neighborhood $O(f)=<f, K,
\epsilon >$ of $f$  where $\epsilon>0$ and $K=\{x_1, ..., x_k\}$
is a finite subset of $X$, there is $i'\in \omega$ such that for
each $i>i'$ and $j\in \{1,...,k\}$ there is $g\in \mathcal{F}_{i}$
such that $g(x_j)\in (f(x_j)-\epsilon, f(x_j)+\epsilon)$.

Consider a sequence $\{Q_i\}_{i\in
\omega}=\{U^{m_1}_{i},...,U^{m_s}_{i}: i\in \omega\}$.

(a). $Q_i \subset \mathcal{U}_{i}$.

(b). $\{\bigcup Q_i: i\in \omega\}$ is a $\gamma$-cover of $X$.

 We thus get $X$ $\models$ $U_{fin}(\Gamma_F, \Gamma)$. By Theorem \ref{th38}, $X$ $\models$ $V$.

\end{proof}

\begin{theorem}\label{th1531} For a separable metrizable space $X$, the following statements are
equivalent:

\begin{enumerate}

\item $C_p(X)$ $\models$ $U_{fin}(\mathcal{S},\mathcal{S})$;

\item $X$ $\models$ $U_{fin}(\Gamma, \Gamma)$ [Hurewicz property];

\item  $C_p(X)$ $\models$ $U_{fin}(\Gamma_x, \Gamma_x)$;

\item  $C_p(X)$ $\models$ $U_{fin}(\mathcal{S}, \Gamma_x)$.

\end{enumerate}

\end{theorem}

\section{$S_{1}(\mathcal{A},\mathcal{A})$}

\begin{definition} A set $A\subseteq C_p(X)$ will
be called {\it $n$-dense} in $C_p(X)$, if for each $n$-finite set
$\{x_1,...,x_n\}\subset X$ such that $x_i\neq x_j$ for $i\neq j$
and an open sets $W_1,..., W_n$ in $\mathbb{R}$ there is $f\in A$
such that $f(x_i)\in W_i$ for $i\in \overline{1,n}$.

\end{definition}

Obviously, that if $A$ is a $n$-dense set of $C_p(X)$ for each
$n\in \omega$ then $A$ is a dense set of $C_p(X)$.

For a space $C_p(X)$ we denote:

$\mathcal{A}_n$
--- the family of a $n$-dense subsets of $C_p(X)$.

If $n=1$, then we denote $\mathcal{A}$ instead of $\mathcal{A}_1$.

\begin{definition} Let $f\in C(X)$. A set $B\subseteq C_p(X)$ will
be called {\it $n$-dense} at point $f$, if for each $n$-finite set
$\{x_1,...,x_n\}\subset X$ and $\epsilon>0$ there is $h\in B$ such
that $h(x_i)\in (f(x_i)-\epsilon, f(x_i)+\epsilon)$ for $i\in
\overline{1,n}$.
\end{definition}

Obviously, that if $B$ is a $n$-dense at point $f$ for each $n\in
\omega$ then $f\in \overline{B}$.

For a space $C_p(X)$ we denote:

$\mathcal{B}_{n,f}$
--- the family of a $n$-dense at point $f$ subsets of $C_p(X)$.

If $n=1$, then we denote $\mathcal{B}_f$ instead of
$\mathcal{B}_{1,f}$.

 Let $\mathcal{U}$ be an open cover of $X$ and $n\in \mathbb{N}$.

$\bullet$ $\mathcal{U}$ is an $n$-cover of $X$ if for each
$F\subset X$ with $|F|\leq n$, there is $U\in \mathcal{U}$ such
that $F \subset U$ \cite{bts1}.

$\bullet$ $\mathcal{O}_n$ --- the family of open $n$-covers of
$X$.

$\bullet$  $S_{1}(\mathcal{O}, \mathcal{O})=S_{1}(\Omega,
\mathcal{O})$ \cite{sch3}.

$\bullet$ $S_{1}(\Omega,
\mathcal{O})=S_{1}(\{\mathcal{O}_n\}_{n\in \mathbb{N}},
\mathcal{O})$ \cite{bts1}.

\begin{theorem}\label{th143} For a   space $X$, the following statements are
equivalent:

\begin{enumerate}

\item  $C_p(X)$ $\models$ $S_{1}(\mathcal{A},\mathcal{A})$;

\item $X$ $\models$ $S_{1}(\mathcal{O}, \mathcal{O})$ [Rothberger
property];

\item $C_p(X)$ $\models$ $S_{1}(\mathcal{B}_f,\mathcal{B}_f)$;

\item  $C_p(X)$ $\models$ $S_{1}(\mathcal{A},\mathcal{B}_f)$;

\item  $C_p(X)$ $\models$ $S_{1}(\mathcal{D},\mathcal{A})$;

\item  $C_p(X)$ $\models$ $S_{1}(\{\mathcal{A}_n\}_{n\in
\mathbb{N}},\mathcal{A})$;

\item $C_p(X)$ $\models$ $S_{1}(\{\mathcal{B}_{n,f}\}_{n\in
\mathbb{N}},\mathcal{B}_f)$;

\item  $C_p(X)$ $\models$ $S_{1}(\{\mathcal{A}_n\}_{n\in
\mathbb{N}},\mathcal{B}_f)$.

\end{enumerate}

\end{theorem}

\begin{proof} $(1)\Rightarrow(2)$.  Let $\{\mathcal{O}_n\}_{n\in \omega}$ be a
sequence of open covers of $X$. We set $A_n=\{f\in C(X):
f\upharpoonright (X\setminus U)=1$ and $f\upharpoonright K=q$ for
some $U\in \mathcal{O}_n$ , a finite set $K\subset U$ and $q\in
\mathbb{Q}\}$. It is not difficult to see that each $A_n$ is
$1$-dense subset of $C_p(X)$  since each $\mathcal{O}_n$ is a
cover of $X$ and $X$ is Tychonoff.

 By the assumption there exist $f_n\in A_n$ such that
$\{f_n : n\in \omega\}\in \mathcal{A}$.

 For each $f_n$ we
take $U_n\in \mathcal{O}_n$ such that
$f_n\upharpoonright(X\setminus U_n)=1$.

 Set $\mathcal{U}=\{ U_n : n\in \omega\}$. For $x\in X$ we consider the basic open neighborhood
of $\bf{0}$ $[x, W]$, where $W=(-\frac{1}{2},\frac{1}{2})$.

 Note that there is $m\in \omega$ such that
$[x, W]$ contains $f_m\in \{f_n : n\in \omega\}$. This means $x\in
 U_m$. Consequently $\mathcal{U}$ is cover
of $X$.

$(2)\Rightarrow(3)$. Let $B_n\in \mathcal{B}_f$ for each $n\in
\omega$. We renumber $\{B_n\}_{n\in \omega}$ as
$\{B_{i,j}\}_{i,j\in \omega}$.  Since $C(X)$ is homogeneous, we
may think that $f=\bf{0}$.  We set
$\mathcal{U}_{i,j}=\{g^{-1}(-1/i, 1/i) : g\in B_{i,j}\}$ for each
$i,j\in \omega$. Since $B_{i,j}\in \mathcal{B}_0$,
$\mathcal{U}_{i,j}$ is an open cover of $X$ for each $i,j\in
\omega$. In case the set $M=\{i\in \omega: X\in \mathcal{U}_{i,j}
\}$ is infinite, choose $g_{m}\in B_{m,j}$ $m\in M$ so that
$g^{-1}(-1/m, 1/m)=X$, then $\{g_m : m\in \omega\}\in
\mathcal{B}_f$.

So we may assume that there exists $i'\in \omega$ such that for
each $i\geq i'$ and $g\in B_{i,j}$ $g^{-1}(-1/i, 1/i)$ is not $X$.

For the sequence $\mathcal{V}_i=\{\mathcal{U}_{i,j} : j\in
\omega\}$ of open covers there exist $f_{i,j}\in B_{i,j}$ such
that $\mathcal{U}_i=\{f^{-1}_{i,j}(-1/i,1/i):  j\in \omega\}$ is a
cover of $X$.  Let $[x, W]$ be any basic open neighborhood of
$\bf{0}$, where $W=(-\epsilon, \epsilon)$, $\epsilon>0$. There
exists $m\geq i'$ and $j\in \omega$  such that $1/m<\epsilon$ and
$x\in f^{-1}_{m,j}(-1/m, 1/m)$. This means $\{f_{i,j}: i,j\in
\omega\}\in \mathcal{B}_f$.

$(3)\Rightarrow(4)$ is immediate.

$(4)\Rightarrow(1)$. Let $A_n\in \mathcal{A}$ for each $n\in
\omega$. We renumber $\{A_n\}_{n\in \omega}$ as
$\{A_{i,j}\}_{i,j\in \omega}$. Renumber the rational numbers
$\mathbb{Q}$ as $\{q_i : i\in \omega\}$.  Fix $i\in\omega$. By the
assumption there exist $f_{i,j}\in A_{i,j}$ such that $\{f_{i,j} :
j\in \omega\}\in \mathcal{A}_{q_i}$ where $q_i$  is the constant
function to $q_i$. Then $\{f_{i,j} : i,j\in \omega\}\in
\mathcal{A}$.

$(1)\Rightarrow(5)$. Since a dense set of $C_p(X)$ is a $1$-dense
set of $C_p(X)$, we have $C_p(X)$ $\models$
$S_{1}(\mathcal{D},\mathcal{A})$.

$(5)\Rightarrow(6)$. Let $D_n \in \mathcal{A}_n$ for each $n\in
\omega$. We renumber $\{D_n\}_{n\in \omega}$ as
$\{D_{i,j}\}_{i,j\in \omega}$. Then $P_j=\{D_{i,j} : i\in
\omega\}$ is a dense subset of $C_p(X)$ for each $j\in \omega$. By
(5), there is $p_j\in P_j$ for each $j\in \omega$ such that
$\{p_j: j\in \omega\}\in \mathcal{A}$. Hence, we have $C_p(X)$
$\models$ $S_{1}(\{\mathcal{A}_n\}_{n\in
\mathbb{N}},\mathcal{A})$.

$(6)\Rightarrow(8)$ is immediate.

$(8)\Rightarrow(2)$. Claim that $X$ $\models$
$S_{1}(\{\mathcal{O}_n\}_{n\in \omega}, \mathcal{O})$. Fix
$\{\mathcal{O}_n\}_{n\in \omega}$. For every $n\in \omega$ a set
$\mathcal{S}_n=\{f\in C(X) : f\upharpoonright (X\setminus U)=1$
and $f(x_i)\in \mathbb{Q}$ for each $i=\overline{1,n}$ for  $U\in
\mathcal{O}_n$ and a finite set $K=\{x_1,..., x_n\}\subset U \}$.
Note that $\mathcal{S}_n\in \mathcal{A}_n$ for each $n\in \omega$.
By $(8)$, there is $f_n\in \mathcal{S}_n$ for each $n\in \omega$
such that $\{f_n: n\in \omega\}\in \mathcal{B}_0$. Then $\{U_n :
n\in \omega \}\in \mathcal{O}$.

$(3)\Rightarrow(7)$ is immediate.

$(7)\Rightarrow(2)$. The proof is analogous to proof of
implication $(8)\Rightarrow(2)$.

\end{proof}

\section{$S_{fin}(\mathcal{A},\mathcal{A})$}

\begin{theorem}\label{th144} For a   space $X$, the following statements are
equivalent:

\begin{enumerate}

\item  $C_p(X)$ $\models$ $S_{fin}(\mathcal{A},\mathcal{A})$;

\item $X$ $\models$ $S_{fin}(\mathcal{O}, \mathcal{O})$ [Menger
property];

\item $C_p(X)$ $\models$ $S_{fin}(\mathcal{B}_f,\mathcal{B}_f)$;

\item  $C_p(X)$ $\models$ $S_{fin}(\mathcal{A},\mathcal{B}_f)$;

\item  $C_p(X)$ $\models$ $S_{fin}(\mathcal{D},\mathcal{A})$;

\item  $C_p(X)$ $\models$ $S_{fin}(\{\mathcal{A}_n\}_{n\in
\mathbb{N}},\mathcal{A})$;

\item $C_p(X)$ $\models$ $S_{fin}(\{\mathcal{B}_{n,f}\}_{n\in
\mathbb{N}},\mathcal{B}_f)$;

\item  $C_p(X)$ $\models$ $S_{fin}(\{\mathcal{A}_n\}_{n\in
\mathbb{N}},\mathcal{B}_f)$.

\end{enumerate}

\end{theorem}

\begin{proof}

The proof is analogous to proof of Theorem \ref{th143}.

\end{proof}

\section{$S_{1}(\mathcal{S},\mathcal{A})$}

\begin{proposition}\label{th222} For a   space $X$, the following statements are
equivalent:

\begin{enumerate}

\item $C_p(X)$ $\models$ $S_{1}(\Gamma_x,\mathcal{B}_f)$;

 \item
$X$ $\models$ $S_{1}(\Gamma_F, \mathcal{O})$.
\end{enumerate}

\end{proposition}

\begin{proof} $(1)\Rightarrow(2)$. Let $\{\mathcal{U}_i\}\subset
\Gamma_F$. For each $i\in \omega$ we consider the
 set $\mathcal{S}_i=\{ f\in C(X) : f\upharpoonright
F(U)=0$ and $f\upharpoonright (X\setminus U)=1$ for $U\in
\mathcal{U}_i \}$.

Since $\mathcal{F}_i=\{F(U): U\in \mathcal{U}_i\}$ is a
$\gamma$-cover of $X$, we have that $\mathcal{S}_i$ converge to
${\bf 0}$, i.e. $\mathcal{S}_i\in \Gamma_0$ for each $i\in
\omega$.

Since  $C_p(X)$ $\models$ $S_{1}(\Gamma_x,\mathcal{B}_f)$, there
is a sequence $\{f_{i}\}_{i\in\omega}$ such that for each $i$,
$f_{i}\in \mathcal{S}_i$, and $\{f_{i} : i\in\omega\}\in
\mathcal{B}_{\bf 0}$.

Consider $\mathcal{V}=\{U_i : U_i\in \mathcal{U}_i$ such that
$f_i\upharpoonright F(U_i)=0$ and $f_i\upharpoonright (X\setminus
U_i)=1\} $. Let $x\in X$ and $W=[x,(-1,1)]$ be a neighborhood of
$\bf{0}$, then there exists $i_0\in \omega$ such that $f_{i_0}\in
W$ .

 It follows that
$x\in U_{i_0}$ and $\mathcal{V}\in \mathcal{O}$. We thus get $X$
$\models$ $S_{1}(\Gamma_F, \mathcal{O})$.

$(2)\Rightarrow(1)$.  Fix $\{S_n : n\in \omega\}\subset \Gamma_0$.
We renumber $\{S_n : n\in \omega\}$ as $\{S_{i,j}: i,j\in
\omega\}$.

For each $i,j\in \omega$ and $f\in S_{i,j}$, we put
$U_{i,j,f}=\{x\in X : |f(x)|<\frac{1}{i+j}\}$, $Z_{i,j,f}=\{x\in X
: |f(x)|\leq\frac{1}{i+j+1}\}$.

Each $U_{i,j,f}$ (resp., $Z_{i,j,f}$) is a cozero-set (resp.,
zero-set) in $X$ with $Z_{i,j,f}\subset U_{i,j,f}$. Let
$\mathcal{U}_{i,j}=\{ U_{i,j,f} : f\in S_{i,j}\}$ and
$\mathcal{Z}_{i,j}=\{ Z_{i,j,f} : f\in S_{i,j}\}$. So without loss
of generality, we may assume $U_{i,j,f}\neq X$ for each $i,j\in
\omega$ and $f\in S_{i,j}$.

We can easily check that the condition $S_{i,j}\in \Gamma_0$
implies that $\mathcal{Z}_{i,j}$ is a $\gamma$-cover of $X$.

Since  $X$ $\models$ $S_{1}(\Gamma_F, \mathcal{O})$ for each $j\in
\omega$ there is a sequence $\{U_{i,j,f_{i,j}} : i\in \omega\}$
such that for each $i$, $U_{i,j,f_{i,j}}\in \mathcal{U}_{i,j}$,
and $\{U_{i,j,f_{i,j}} : i\in \omega\}\in \mathcal{O}$. Claim that
$\{f_{i,j}:i,j\in \omega\}\in \mathcal{B}_0$. Let $x\in X$,
$\epsilon>0$, and $W=[x,(-\epsilon, \epsilon)]$ be a base
neighborhood of $\bf{0}$, then there exists $j'\in \omega$ such
that $\frac{1}{1+j'}<\epsilon$. It follow that there exists $i'$
such that $f_{i',j'}(x)\in (-\epsilon, \epsilon)$. So $C_p(X)$
$\models$ $S_{1}(\Gamma_x,\mathcal{B}_f)$.

\end{proof}

\begin{theorem}\label{th173} For a   space $X$, the following statements are
equivalent:

\begin{enumerate}

\item  $C_p(X)$ $\models$ $S_{1}(\mathcal{S},\mathcal{A})$ and is
sequentially separable;

\item $X$ $\models$ $S_{1}(\Gamma_F, \mathcal{O})$, $X$ $\models$
$V$;

\item $C_p(X)$ $\models$ $S_{1}(\Gamma_x,\mathcal{B}_f)$ and is
sequentially separable;

\item  $C_p(X)$ $\models$ $S_{1}(\mathcal{S},\mathcal{B}_f)$ and
is sequentially separable.

\end{enumerate}

\end{theorem}

\begin{proof} $(1)\Rightarrow(2)$. Let $\{\mathcal{V}_i: i\in \omega\}\subset \Gamma_F$ and
$\mathcal{S}=\{h_m: m\in \omega\}$ be a countable sequentially
dense subset of $C_p(X)$. For each $i\in \omega$ we consider a
countable sequentially dense subset $\mathcal{S}_i$ of $C_p(X)$
and $\mathcal{U}_i=\{ U^{m}_i: m\in \omega\}\subset \mathcal{V}_i$
where

$\mathcal{S}_i=\{ f^m_i\in C(X) : f^m_i\upharpoonright
F(U^{m}_i)=h_m$ and $f^m_i\upharpoonright (X\setminus U^{m}_i)=1$
for $m \in \omega \}$.

Since $\mathcal{F}_i=\{F(U^{m}_i): m\in \omega\}$ is a
$\gamma$-cover of zero subsets of $X$ and $\mathcal{S}$ is a
countable sequentially dense subset of $C_p(X)$, we have that
$\mathcal{S}_i$ is a  countable sequentially dense subset of
$C_p(X)$ for each $i\in \omega$.  Let $h\in C(X)$, there is a
sequence $\{h_{m_s}: s\in \omega\}\subset \mathcal{S}$ such that
$\{h_{m_s}\}_{s\in \omega}$ converge to $h$.
 Let $K$
be a finite subset of $X$, $\epsilon>0$ and $W=<h, K,\epsilon>$ be
a base neighborhood of $h$, then there is a number $m_0$ such that
$K\subset F(U^{m}_i)$ for $m>m_0$ and $h_{m_s}\in W$ for
$m_s>m_0$. Since $f^{m_s}_i\upharpoonright K=
h_{m_s}\upharpoonright K$ for each $m_s>m_0$, $f^{m_s}_i\in W$ for
each $m_s>m_0$. It follows that a sequence $\{f^{m_s}_i\}_{s\in
\omega}$ converge to $h$.

By $C_p(X)\in S_{1}(\mathcal{S},\mathcal{A})$, there is a set
$\{f^{m(i)}_{i}: i\in\omega\}$ such that for each $i$,
$f^{m(i)}_{i}\in \mathcal{S}_i$, and $\{f^{m(i)}_{i}: i\in\omega
\}$ is an element of $\mathcal{A}$.

Consider a set $\{U^{m(i)}_{i}: i\in \omega\}$.

(a). $U^{m(i)}_{i}\in \mathcal{U}_{i}$.

(b). $\{U^{m(i)}_{i}: i\in \omega\}$ is a cover of $X$.

Let $x\in X$ and $U=<$ $\bf{0}$ $, x, \frac{1}{2}>$ be a base
neighborhood of $\bf{0}$, then there is
$f^{m(i)_{j_0}}_{i_{j_0}}\in U$ for some $j_0\in \omega$. It
follows that $x\in  U^{m(i)_{j_0}}_{i_{j_0}}$. We thus get $X$
$\models$ $S_{1}(\Gamma_F, \mathcal{O})$.

$(2)\Leftrightarrow(3)$. By Proposition \ref{th222}.

$(3)\Rightarrow(4)$ is immediate.

$(4)\Rightarrow(1)$. Let $S_n\in \mathcal{S}$ for each $n\in
\omega$. We renumber $\{S_n\}_{n\in \omega}$ as
$\{S_{i,j}\}_{i,j\in \omega}$. Renumber the rational numbers
$\mathbb{Q}$ as $\{q_i : i\in \omega\}$.  Fix $i\in\omega$. By the
assumption there exist $f_{i,j}\in S_{i,j}$ such that $\{f_{i,j} :
j\in \omega\}\in \mathcal{B}_{q_i}$ where $q_i$  is the constant
function to $q_i$. Then $\{f_{i,j} : i,j\in \omega\}\in
\mathcal{A}$.

\end{proof}

\begin{theorem}\label{th173} For a separable metrizable space $X$, the following statements are
equivalent:

\begin{enumerate}

\item  $C_p(X)$ $\models$ $S_{1}(\mathcal{S},\mathcal{A})$;

\item $X$ $\models$ $S_{1}(\Gamma_F, \mathcal{O})$;

\item $C_p(X)$ $\models$ $S_{1}(\Gamma_x,\mathcal{B}_f)$;

\item  $C_p(X)$ $\models$ $S_{1}(\mathcal{S},\mathcal{B}_f)$.

\end{enumerate}

\end{theorem}

\section{Critical cardinalities}

For a collection $\mathcal{J}$ of spaces $C_p(X)$, let
$nonC_p(\mathcal{J})$ denote the minimal cardinality for $X$ which
$C_p(X)$ is not a member of $\mathcal{J}$.

 The critical cardinalities in the Scheepers Diagram \cite{tss1} are equal to the critical cardinalities of selectors for
sequences of countable dense and countable sequentially subsets of
$C_p(X)$.

\begin{theorem} For a collection $C_p(X)$ of all
real-valued continuous functions, defined on  Tychonoff spaces $X$
with $iw(X)=\aleph_0$,

(1) $nonC_p(S_{1}(\mathcal{D},\mathcal{S}))=\mathfrak{p}$.

(2)
$nonC_p(S_{1}(\mathcal{S},\mathcal{S}))=nonC_p(U_{fin}(\mathcal{S},\mathcal{S}))=\mathfrak{b}$.

(3)
$nonC_p(S_{fin}(\mathcal{D},\mathcal{D}))=nonC_p(S_{1}(\mathcal{S},\mathcal{D}))=nonC_p(S_{1}(\mathcal{S},\mathcal{A}))=\mathfrak{d}$.

$nonC_p(U_{fin}(\mathcal{S},\mathcal{D}))=nonC_p(U_{fin}(\mathcal{S},\mathcal{A}))=
nonC_p(S_{fin}(\mathcal{S},\mathcal{D}))=\mathfrak{d}$.

(4) $nonC_p(S_{1}(\mathcal{D},\mathcal{D})) =
nonC_p(S_{1}(\mathcal{A},\mathcal{A}))=cov(\mathcal{M})$.

\end{theorem}

\newpage

We can summarize the relationships between considered notions in
next diagrams.

\begin{center}
\ingrw{90}{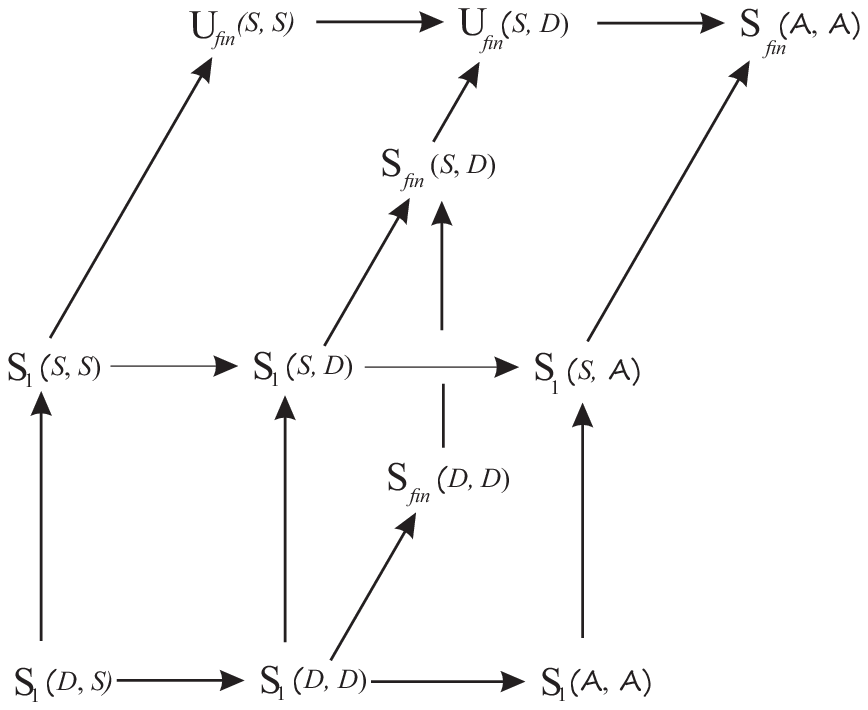}

\medskip

Fig.~2. The Diagram of selectors for sequences of dense sets of
$C_p(X)$.

\end{center}

\bigskip

\begin{center}
\ingrw{90}{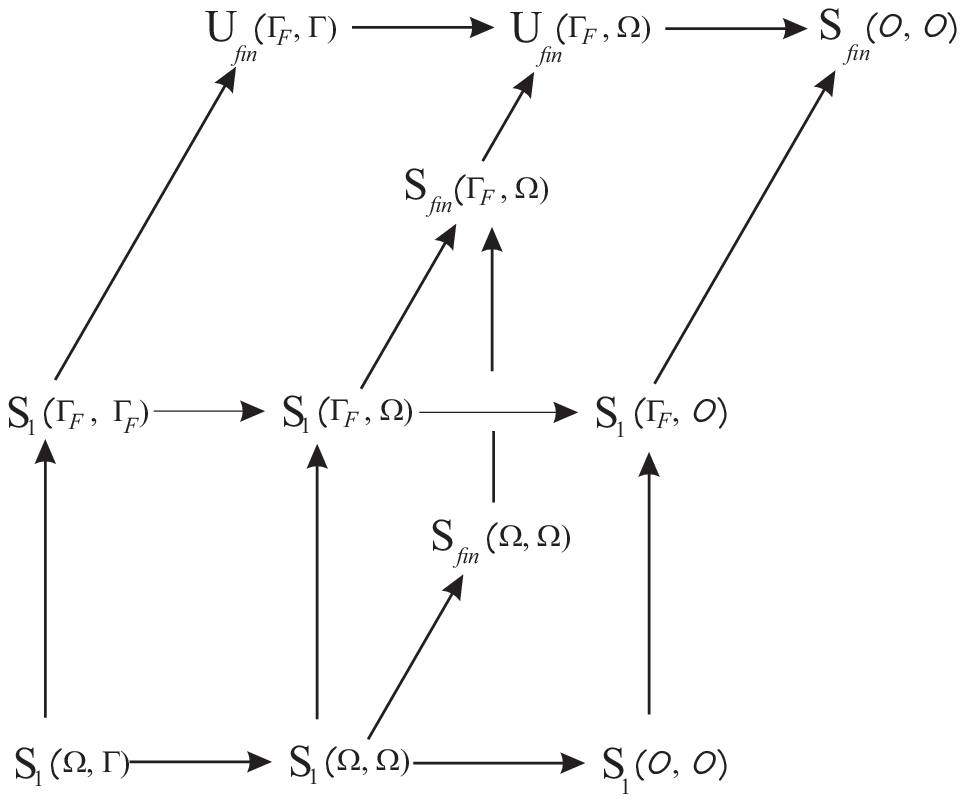}

\medskip

Fig.~3. The Diagram of selection principles for metrizable
separable space $X$ corresponding to selectors for sequences of
dense sets of $C_p(X)$.

\end{center}

\bigskip

{\bf Acknowledgment.} The author express gratitude to Boaz Tsaban
for useful discussions.

\bigskip

\bibliographystyle{model1a-num-names}
\bibliography{<your-bib-database>}







\end{document}